\documentclass[11pt]{article}

\usepackage{amsmath}
\usepackage{amsthm}
\usepackage{amsfonts}
\usepackage{amssymb}
\usepackage{latexsym}

\DeclareMathOperator{\WI}{W}
\DeclareMathOperator{\asc}{asc}
\DeclareMathOperator{\des}{des}

\usepackage[noautoscale]{youngtab}

\usepackage{fullpage}

\usepackage{color, colortbl}
\usepackage{array}
\usepackage{multirow}

\usepackage{tikz}

\numberwithin{equation}{section}

\newtheorem{thm}{Theorem}[section]
\newtheorem{lem}[thm]{Lemma}
\newtheorem{prop}[thm]{Proposition}
\newtheorem{cor}[thm]{Corollary}
\newtheorem{prob}[thm]{Problem}

\newtheorem{rem}[thm]{Remark}

\newtheorem{exam}[thm]{Example}

\DeclareMathOperator{\level}{level}
\DeclareMathOperator{\levMax}{levMax}
\DeclareMathOperator{\levSum}{levSum}
\DeclareMathOperator{\Ca}{\mathcal{C}}

\newcommand{\md}{{\rm mod}}

\usepackage[colorlinks=true,
linkcolor=webgreen,
filecolor=webbrown,
citecolor=webgreen]{hyperref}

\definecolor{webgreen}{rgb}{0,.5,0}
\definecolor{webbrown}{rgb}{.6,0,0}

\usepackage[symbol]{footmisc}

\begin{document}

\title{Patterns in Multi-dimensional Permutations\thanks{  Shaoshi Chen, Hanqian Fang and Candice X.T.\ Zhang were partially supported by the NSFC grant (No. 12271511), the CAS Funds of
the Youth Innovation Promotion Association (No.\ Y2022001), the Strategic Priority Research Program of the Chinese Academy of Sciences (No.\ XDB0510201),
and the National Key R\&D Program of China (No.\ 2023YFA1009401). Candice X.T.\ Zhang was partially supported by the Postdoctoral Fellowship Program of CPSF under Grant Number: GZC20241868.
}}
\author{Shaoshi Chen\footnote{KLMM, Academy of Mathematics and Systems Science, Chinese Academy of Sciences, Beijing 100190, P. R. China. Email: schen@amss.ac.cn.}, Hanqian Fang\footnote{KLMM, Academy of Mathematics and Systems Science, Chinese Academy of Sciences, Beijing 100190, P. R. China.  Email: fanghanqian22@mails.ucas.ac.cn.}, Sergey Kitaev\footnote{Department of Mathematics and Statistics, University of Strathclyde, 26 Richmond Street, Glasgow G1 1XH, United Kingdom. Email: sergey.kitaev@strath.ac.uk.}\ \ and Candice X.T. Zhang\footnote{KLMM, Academy of Mathematics and Systems Science, Chinese Academy of Sciences, Beijing 100190, P. R. China. Email: zhangxutong@amss.ac.cn.}}

\date{\today}

\maketitle

\noindent\textbf{Abstract.}
In this paper, we propose a general framework that extends the theory of permutation patterns to higher dimensions and unifies several combinatorial objects studied in the literature. Our approach involves introducing the concept of a ``level'' for an element in a multi-dimensional permutation, which can be defined in multiple ways. We consider two natural definitions of a level, each establishing connections to other combinatorial sequences found in the Online Encyclopedia of Integer Sequences (OEIS).

Our framework allows us to offer combinatorial interpretations for various sequences found in the OEIS, many of which previously lacked such interpretations. As a notable example, we introduce an elegant combinatorial interpretation for the Springer numbers: they count weakly increasing 3-dimensional permutations under the definition of levels determined by maximal entries.\\

\noindent {\bf AMS Classification 2010:}  05A05; 05A15


\tableofcontents


\section{Introduction}\label{sec-intro}

The field of permutation patterns is a popular and rapidly growing area of research \cite{Kitaev2011}. In the literature, there have been successful attempts to extend the theory of permutation patterns to higher dimensions through the introduction of \emph{multi-dimensional permutations}, which can be defined in several ways. In \cite{ZG2007,AM2010,AKLPT}, the authors defined a multi-dimensional permutation as a sequence of one-line permutations, with the first one being the natural order permutation $12\cdots n$.
Under this definition, Zhang and Gildea \cite{ZG2007} studied the hierarchical decomposition of multi-dimensional permutations by considering them as alignment structures across sequences or visualizing them as the result of permuting $n$ hypercubes of certain dimensions; Asinoeski and Mansour \cite{AM2010} examined the multi-dimensional generalization of separable permutations, providing both a characterization and enumeration of these permutations; the authors in \cite{AKLPT} generalized the concept of mesh patterns in multi-dimensional permutations. Moreover, the authors in \cite{EriLin, Linial2014, Linial2018, Linial2019} defined multi-dimensional permutations from the perspective of high-dimensional analogues of permutation matrices and derived their combinatorial properties.
For our purposes, we define a multi-dimensional permutation as a collection of permutations of the same length, expressed in one-line notation and arranged in matrix form (see Section~\ref{prelim-sec} for a formal definition). The elements of these permutations correspond to the columns.

In this paper, we propose a general framework based on the notion of a pattern in a multi-dimensional permutation. This framework allows us to provide new combinatorial interpretations for various sequences studied in the literature and to unify several combinatorial objects appearing in the {\em Online Encyclopedia of Integer Sequences} ({\em OEIS}) \cite{oeis}. For example, the $d$-dimensional hoe permutations, for $d=4,5,6$ give, respectively, sequences  A023000, A135518, A218734 in \cite{oeis} with no combinatorial interpretation (see Section~\ref{hoe-section}), while $c$-bounded 3-dimensional permutations under a canonical representation for $c=2,3,4$ correspond to sequences A008776, A002023, A235702 in \cite{oeis}, respectively, with no common interpretation (see Section~\ref{c-bounded-perms-sec}).

To extend permutation patterns to the multi-dimensional setting, we introduce the notion of a {\em level} of an element in a multi-dimensional permutation, which is a function mapping each column to a non-negative integer. The notion of a level can be defined in multiple ways, two of which, $\levMax$ and $\levSum$, are studied in this paper, each revealing interesting connections to known combinatorial objects. $\levMax(C)$ returns the maximal entry in a column $C$, while $\levSum(C)$ sums the entries of $C$. It is noteworthy that some of the involved research questions for $\levMax$ have trivial solutions for $\levSum$, and conversely. This observation holds generally: certain challenging research problems under one definition of level may have straightforward solutions under the other, and vice versa.

Our most intriguing results focus on enumerating multi-dimensional permutations under $\levMax$. One notable finding is the connection between certain 3-dimensional permutations and sequence A008971 in~\cite{oeis}, which counts (2-dimensional) permutations of length $n$ with $k$ increasing runs of length at least 2, as discussed in \cite[Chapter 10]{DB1962}. This leads us to introduce, in Theorem~\ref{thm-R(x,y)}, an alternative form for the exponential generating function of these numbers, as provided in \cite[A008971]{oeis} and derived by Zhuang \cite{Zhuang2016} through the application of Gessel's run theorem \cite{Gessel1977}.

Another significant result for multi-dimensional permutations under $\levMax$, demonstrated in Section~\ref{WI-perms}, is that {\em weakly increasing} 3-dimensional permutations are enumerated by the {\em Springer numbers}.  This important class of numbers, rooted in representation theory, encompasses various combinatorial interpretations. For example, Arnol'd~\cite{Arnold} showed that the Springer numbers enumerate a signed-permutation analogue of the alternating permutations, involving the notion of ``snakes of type $B_n$''. Several other combinatorial objects are also enumerated by the Springer numbers: Weyl chambers in the principal Springer cone of the Coxeter group $B_n$~\cite{Springer1971}, topological types of odd functions with $2n$ critical values~\cite{Arnold}, labeled ballot paths~\cite{Chen-et-al}, certain classes of complete binary trees and plane rooted forests~\cite{J-V2014}, and up-down permutations of even length fixed under reverse and complement~\cite{HKZ2024}. Additionally, the Springer numbers are studied in relation to the classical moment problem~\cite{Sokal2020}.

This paper is organized as follows. In Section~\ref{prelim-sec}, we introduce all necessary definitions and notations, including levels and patterns in multi-dimensional permutations.  Section~\ref{levMax-sec} is devoted to the study of properties of patterns in multi-dimensional permutations when levels are defined by $\levMax$.
In particular, we  explicitly derive exponential generating functions for  weakly increasing and unimodal 3-dimensional permutations under $\levMax$. Additionally, we investigate real-rootedness of certain relevant polynomials.
In Section \ref{levSum-sec}, we define and study several types of special permutations when the levels are defined by  $\levSum$. Finally, in Section~\ref{open-questions-sec}, we provide concluding remarks and suggest some problems for further research.

\section{Multi-dimensional permutations, levels, and patterns}\label{prelim-sec}
{In this section, we will introduce the definitions and notations related to multi-dimensional permutations. Specifically, we will define the notion of ``level'' to analyze ``pattern'' occurrences in multi-dimensional permutations, which will be detailed in Subsections \ref{levels-def-sec} and \ref{pattern-in-multi-per}, respectively.}

Firstly, we follow \cite{AKLPT} to define the notion of a multi-dimensional permutation. Due to technical reasons (which will become clear after we introduce the notion of a level, specifically $\levSum$), we need to assume that permutations in the symmetric group $S_n$ involve elements from $\{0,1,\ldots,n-1\}$ instead of the more typical $\{1,2,\ldots,n\}$.

Let $\pi = \pi_1 \pi_2 \dots \pi_n$ be a permutation of length $n$ ($n$-permutation) in $S_n$.  As written, $\pi$ is in one-line notation, while its two-line notation is \[\pi=\left(\begin{array}{cccc}0&1&\dots&n-1\\ \pi_1 & \pi_2 & \dots & \pi_n\\ \end{array}\right).\]
A \textit{$d$-dimensional permutation $\Pi$ of length $n$} (or \textit{$d$-dimensional $n$-permutation}, or \textit{$(d,n)$-permutation}) is an ordered $(d-1)$-tuple $(\pi^2,\pi^3, \dots , \pi^d)$ of $n$-permutations where for each $2\leq i\leq d$, $\pi^i=\pi_1^i\pi_2^i\dots\pi_n^i\in S_n$. For example, $(201,201,120)$ is a 4-dimensional permutation of length~3. We let $S^d_n$ denote the set of $d$-dimensional permutations of length $n$. Note that $S^2_n$ corresponds naturally to $S_n$, hence ``usual'' permutations are 2-dimensional permutations.  We also generalize two-line notation to $d$-line notation and we write
\renewcommand{\arraystretch}{1.25}
\[\Pi=\left(\begin{array}{cccc}
0&1&\dots&n-1\\
\pi^2_1 & \pi^2_2 & \dots & \pi^2_n\\
 \pi^3_1 & \pi^3_2 & \dots & \pi^3_n\\
  \vdots & \vdots & \ddots & \vdots\\
   \pi^d_1 & \pi^d_2 & \dots & \pi^d_n\\ \end{array}\right)=\left(\begin{array}{cccc}\pi^1_1&\pi^1_2&\dots&\pi^1_n\\
\pi^2_1 & \pi^2_2 & \dots & \pi^2_n\\
 \pi^3_1 & \pi^3_2 & \dots & \pi^3_n\\
  \vdots & \vdots & \ddots & \vdots\\
   \pi^d_1 & \pi^d_2 & \dots & \pi^d_n\\ \end{array}\right),\]
so that $\Pi$ corresponds naturally to a $d\times n$ matrix.  It is also helpful to let $\pi^1$ denote the permutation $01\dots (n-1)$ so that we can succinctly write \[\Pi=\left\{\pi^i_j\right\}_{\begin{subarray}{l} 1 \leq i\leq d \\ 1\le j\le n\end{subarray}}.\]  Motivated by two-line notation, we say that the columns of this matrix represent the \textit{elements of $\Pi$} which we denote by  $\Pi_i$.  In particular, we write $\Pi=\Pi_1\Pi_2\dots \Pi_n$ where $\Pi_i$ is the $d$-tuple $({i-1},\pi^2_i, \pi^3_i,\ldots, \pi^d_i)^T$.  For example, if $\Pi = (\pi^1, \pi^2, \pi^3)$ is a $3$-dimensional permutation of length $5$ with $\pi^2 = 01423$ and $\pi^3 = 40132$, then we write
\renewcommand{\arraystretch}{1}
\begin{equation}\label{3-dim-perm-example}
\Pi=\left(\begin{array}{c}\pi^1\\ \pi^2\\ \pi^3 \end{array}\right)=\left(\begin{array}{ccccc}
0&1&2&3&4\\
0&1&4&2&3\\
4&0&1&3&2\\
\end{array}\right)
\end{equation}
or $\Pi=\Pi_1\Pi_2\Pi_3\Pi_4\Pi_5$ where $\Pi_1=(0,0,4)^T$, $\Pi_2=(1,1,0)^T$, $\Pi_3=(2,4,1)^T$, $\Pi_4=(3,2,3)^T$, $\Pi_5=(4,3,2)^T$. However, throughout this paper, we usually omit $\pi^1$, so $d$-dimensional permutations are typically represented by $d-1$ rows and slightly abusing the notation, the $i$-th element (the $i$-th column with $d-1$ elements) is still denoted by $\Pi_i$. This will never cause any confusion. Furthermore, we sometimes concern about the case of $\pi^2=01\ldots(n-1)$. Specifically, given a subset $T$ of $S_n^d$, let $\Ca(T)$ denote the subset of $T$ consisting of those \textit{canonical permutations}, i.e., the permutations $\Pi=(\pi^2,\pi^3,\ldots,\pi^d)$ with $\pi^2=01\ldots(n-1)$.

\subsection{Levels}\label{levels-def-sec}

In considering the extension of results on permutation patterns to the multi-dimensional case, similar to previous work such as \cite{AM2010, AKLPT, ZG2007}, we encounter the need to compare two elements in a multi-dimensional permutation.
A natural way to compare two elements is via the notion of a {\em level} of an element $\Pi_i$ in $\Pi\in S_n^d$,  denoted by $\level(\Pi_i)$, which is simply a function mapping $\Pi_i$ to a non-negative integer. We let $\level(\Pi)=(\level(\Pi_1),\ldots,\level(\Pi_n))$ be the {\em vector of levels} of $\Pi$. Furthermore, we let $\level_{\min}(\Pi)$ (resp. $\level_{\max}(\Pi)$) be the minimum (resp., maximum) entry in $\level(\Pi)$.

The introduction of levels allows for the comparison of elements within a multi-dimensional permutation, extending and generalizing the classical theory of permutation patterns to higher dimensions. It is important to note that if we define the level in such a way that any two distinct elements in a multi-dimensional permutation correspond to distinct levels (for instance, by defining the level vector of a permutation $\Pi=(\pi^2,\pi^3,\ldots,\pi^d)$ to be $\pi^2$), then such definitions are equivalent across our framework and therefore lacks novelty. Consequently, we focus on level functions that may lead to some repeated entries in the level vectors of certain permutations, as these present more interesting and non-trivial structures. Specifically speaking, in this paper, we focus on two ways to define the function $\level$:
\begin{itemize}
\item $\levMax(\Pi_i)$ is the maximum entry in $\Pi_i$, and
\item $\levSum(\Pi_i)$ is the sum of all $d-1$ entries in $\Pi_i$.
\end{itemize}

For example, for the 4-dimensional permutation in (\ref{3-dim-perm-example}) (with the increasing top row removed), $\levMax(\Pi)=(4,1,4,3,4)$,  $\levMax(\Pi_2)=\levMax_{\min}(\Pi)=1$, $\levMax_{\max}(\Pi)=4$,  $\levSum(\Pi)=(4,2,7,8,9)$, $\levSum_{\min}(\Pi)=2$,  and $\levSum_{\max}(\Pi)=9$.

Note that for $\levSum$, each level in $\{0,1,\ldots, (d-1)(n-1)\}$ is attainable among all permutations in $S_n^d$. This is the reason we set the minimum element in a permutation in $S_n$ to be 0 rather than 1. Otherwise, for example, level~1 would never be achievable for $d\geq3$, along with some other levels depending on the value of~$d$.

\subsection{Patterns}\label{pattern-in-multi-per}
In this subsection, we will review the classical pattern occurrences in the 2-dimensional case and then generalize this concept to the multi-dimensional case.

In the 2-dimensional case, a permutation $\pi_1\pi_2\dots\pi_n\in S_n$ {\em contains} an occurrence of a {\em pattern} $p=p_1p_2\dots p_k\in S_k$  if there is a subsequence $\pi_{i_1}\pi_{i_2}\dots\pi_{i_k}$ such that $\pi_{i_j}<\pi_{i_m}$ if and only if $p_j<p_m$. A permutation {\em avoids} a pattern if it contains no occurrences of the pattern.
To match the notation used in the permutation patterns literature, such as in \cite{Kitaev2011}, we assume that a pattern consist of the elements in $\{1,2,\ldots,k\}$ rather than $\{0,1,\ldots,k-1\}$. For example, the permutation $\pi=21034$ avoids the pattern $231$ while $\pi$ contains three occurrences of the pattern 123 (the subsequences 234, 134, and 034).

There are myriad ways to introduce the concept of pattern occurrences in multi-dimensional permutations. For example, we can view such a permutation as a matrix and explore occurrences of various 2-dimensional patterns, similar to previous work \cite{JuSeo2012,KitManVel}  with binary matrices. However, our focus lies in the straightforward extension (or generalization) of the 2-dimensional case to multi-dimensions by seeking pattern occurrences in vectors of permutation levels. This approach allows us to compare any two elements in a multi-dimensional permutation, leading us to introduce the notion of a level. Additionally, in contrast to the two-dimensional case where for each pair $(a,b)$ of elements either $a<b$ or $a>b$, in the case of  $d$-dimensional permutations, $d\geq 3$, it is also possible that $a=b$. Therefore, we study occurrences of patterns in sequences using the alphabet $\{0,1,\ldots\}$~\cite{HeuMan}, bearing in mind that the same word can appear multiple times, while some other words may never appear.

Formally, a permutation $\Pi=\Pi_1\dots\Pi_n\in S_n^d$ {\em contains} an occurrence of a {\em pattern} $p=p_1p_2\dots p_k$ (which is a word over an alphabet $A=\{1,\ldots,s\}$ for some $s\leq k$ and each letter in $A$ occurs in $p$)  if there is a subsequence $\Pi_{i_1}\Pi_{i_2}\dots\Pi_{i_k}$ such that $\level(\Pi_{i_j})<\level(\Pi_{i_m})$ (resp., $\level(\Pi_{i_j})=\level(\Pi_{i_m})$) if and only if $p_j<p_m$ (resp., $p_j=p_m$). We are particularly interested in {\em consecutive patterns}, in occurrences of which elements stay next to each other (i.e., $i_2=i_1+1$, $i_3=i_2+1$, etc in the definition of an occurrence of a pattern). To distinguish consecutive patterns, we underline them, like $\underline{11}$, $\underline{12}$, $\underline{21}$, $\underline{121}$, $\underline{123}$, etc, precisely as it is done in \cite{Kitaev2011}. We note that occurrences of $\underline{11}$ are also known in the literature as {\em plateaux} or {\em levels} in words; however,  we reserve the term  ``level'' for another purpose. Additionally, occurrences of the patterns $\underline{12}$ and $\underline{21}$ are referred to in the literature as {\em ascents} and {\em descents}, respectively, with the respective numbers denoted by $\asc$ and $\des$.
Let $p(\Pi)$ denote the number of occurrences of a pattern $p$ in a permutation $\Pi$. For example, for the 4-dimensional permutation $\Pi$ in (\ref{3-dim-perm-example}), we have
\begin{itemize}
\item under $\levMax$: $\levMax(\Pi)=(4,1,4,3,4)$, $\underline{11}(\Pi)=0$, $\underline{12}(\Pi)=\asc(\Pi)=2$,  $\underline{21}(\Pi)=\des(\Pi)=2$, $212(\Pi)=4$, $123(\Pi)=122(\Pi)=1$ and $321(\Pi)=0$.
\item under $\levSum$: $\levSum(\Pi)=(4,2,7,8,9)$, $\underline{11}(\Pi)=0$, $\underline{12}(\Pi)=\asc(\Pi)=3$,  $\underline{21}(\Pi)=\des(\Pi)=1$, $1234(\Pi)=2$ and $213(\Pi)=3$.
\end{itemize}
For any fixed pattern $p$ and chosen notion of level, $S^d_n(p)$ denotes the set of permutations in $S^d_n$ that avoid the pattern $p$.


\section{$\levMax$: levels defined by maximal entries}\label{levMax-sec}

In this section, we will explore several patterns in multi-dimensional permutations under $\levMax$. Our results concerning $\levMax$ primarily focus on three-dimensional permutations, including permutations with $k$ levels repeated, weakly increasing permutations, and unimodal permutations. Although generalizing from three dimensions to higher dimensions is challenging for most enumeration problems, we have nevertheless obtained some results regarding general multi-dimensional permutations, such as the enumeration of a special class of unimodal permutations which we refer to as ``hoe permutations'' in $S_n^d$ and the total number of plateaux among all permutations in $S_n^d$.

\subsection{Permutations in $S_n^3$ with $k$ levels repeated}\label{k-lev-rep-sec}
In this subsection, we consider the enumeration of those canonical permutations in $S_n^3$ with exactly $k$ levels repeated, denoted by $R_{n,\,k}$.
That is,
\[R_{n,\,k}:=| \{\Pi\in \Ca(S_n^3) : \text{there are exactly $k$ distinct levels repeated twice in }\Pi \}|. \]
For example, the permutation \[\left(\begin{array}{ccccccc}0&1&2&3&4&5&6\\ 1&2&4&5&3&0&6\\ \end{array}\right)\in \Ca(S_7^3)\]
contributes to the count of $R_{7,\,2}$ since levels 4 and 5 are the only ones repeated in it.

In the study of these numbers, we derive their recurrence relations and bivariant exponential generating functions. Furthermore, we obtain the real-rootedness of the generating functions of these numbers with $n$ fixed. Notably, the research of $R_{n,\,k}$ is not only intriguing in its own right but also plays a crucial role in deriving enumeration results for weakly increasing and unimodal permutations, which will be discussed in the following subsections.

\subsubsection{Recurrence relations}\label{subsec-R-rec}

Our first goal is to derive the recurrence relations for $R_{n,\,k}$ in a purely combinatorial manner.
This recurrence relation was, in fact, first obtained by David and Barton \cite[page 163]{DB1962} in their study of increasing runs of permutations. Using techniques of context-free grammar \cite{Chen1993}, Ma \cite{MA2012} also derived the same recurrence relation.

\begin{thm}\label{thm-rec-R(n,k)}
For $n\geq1$ and $k\geq1$, we have
	\begin{align}\label{eq-rec-R(n,k)}
		R_{n,\,k}=(2k+1) R_{n-1,\,k}+(n-2k+1)R_{n-1,\,k-1}
	\end{align}
 	with $R_{n,\,0}=1$ for any $n\ge 0$ and $R_{0,\,k}=0$ for any $k\ge 1$.
\end{thm}

\begin{proof}
		Consider two elements at the same level $\ell$ as a pair. We refer to this pair as being on level~$\ell$.
		Then, there are three ways to obtain a permutation $\Pi\in \Ca(S_n^3)$ with exactly $k$ pairs from a permutation $\Pi'\in \Ca(S_{n-1}^3)$ by inserting two ``$n-1$" entries: \\[-3mm]
		
		\noindent
		{\bf Case 1:}
		If we insert the element $(n-1,n-1)^T$ at the end of the permutation~$\Pi'$, then the number
		of pairs of repeated levels remains unchanged, resulting in $R_{n-1,\,k}$ choices for selecting $\Pi'$. This contributes the term $R_{n-1,\,k}$ to the right-hand side of Equation~\eqref{eq-rec-R(n,k)}. \\[-3mm]
		
		\noindent
		{\bf Case 2:}
		Choose an element $(x,y)^T$ and replace it with $(x,n-1)^T$. Then insert the element $(n-1,y)^T$ at the end of the resulting permutation. There are two cases to choose the element~$(x,y)^T$: \\[-3mm]
		
		\noindent
	    \quad{\bf Case 2(a):}
		If the level of the chosen element $(x,y)^T$ is repeated in $\Pi'$, the insertion removes one pair on level $\max\{x,y\}$ and creates a new pair on level $n-1$, resulting in $\Pi'$ having exactly $k$ pairs.
		The number of choices for the element $(x,y)^T$ is $2k$.
		This contributes $2kR_{n-1,\,k}$ to the right-hand side of Equation~\eqref{eq-rec-R(n,k)}.  \\[-3mm]

		\noindent
		\quad{\bf Case 2(b):}
		If the element $(x,y)^T$ is at a unique level, then this insertion creates a new pair ``$(x,n-1)^T$ and $(n-1,y)^T$'', requiring $\Pi'$ to have exactly $k-1$ pairs. The number of choices for the element $(x,y)^T$ is $n-1-2(k-1)$.
		This contributes
		$(n-1-2(k-1))R_{n-1,\,k-1}$ to the right-hand side of Equation~\eqref{eq-rec-R(n,k)}. \\[-3mm]
		
		By combining the disjoint cases discussed above, we derive Equation~\eqref{eq-rec-R(n,k)}. Moreover, it is clear that $R_{n,\,0}=1$, and $R_{0,\,k}=0$.
\end{proof}

	The special cases for $k=1,2,3$, match sequences in the OEIS \cite{oeis}:
	\begin{equation*}
		\begin{array}{clc}
			R_{n,\,1}=&\frac{1}{4}(3^n-2n-1)&    \cite[A000340]{oeis} \\[8pt]
			R_{n,\,2}=&\frac{1}{16}\left(5^n-(2n-1)3^n+2n^2-2n-2\right)&    \cite[A000363]{oeis} \\[8pt]
			R_{n,\,3}=&\frac{1}{192}\left(3\cdot7^n-(6n-9)5^n+(6n^2-18n+3)3^n-4n^3+18n^2-8n-15\right)&  \cite[A000507]{oeis}
		\end{array}		
	\end{equation*}
	Note that for $k=1$, we provide the first known combinatorial interpretation for this sequence. Additionally, for $k=2$ and $k=3$, there are known combinatorial interpretations for the respective sequences. These interpretations correspond to the number of $n$-permutations with exactly 2 or 3 increasing runs of length at least 2, respectively. In general, $R_{n,\,k}$ is given by A008971 in \cite{oeis}, the triangle read by rows, where the $(n,k)$-th element is the number of $n$-permutations with $k$ increasing runs of length at least 2, see \cite{DB1962,MA2012} for the details. This follows easily from the fact (by inserting element $n$ in an $(n-1)$-permutation) noted in~\cite[A008971]{oeis} that the sequence satisfies exactly the same recurrence relation as Equation (\ref{eq-rec-R(n,k)}). A bijective proof of the noted connection can be readily obtained by matching Cases 1 and 2(a) with the insertion of $n$ at the beginning of the permutation or at the end of an increasing run, while matching Case 2(b) with the remaining possibilities to insert $n$.
	
\subsubsection{Generating functions}\label{subsec-R-GF}
We are now prepared to provide the explicit formula for the bivariate exponential generating function of $R_{n,\,k}$.
While Zhuang \cite{Zhuang2016} has provided a proof of this formula using Gessel's run theorem \cite{Gessel1977}, we demonstrate it by solving the associated partial differential equations.

Let
	\begin{align}\label{eq-def-R(x,y)}
		\mathcal{R}(x,y):=\sum_{n,\,k\ge0}\frac{R_{n,\,k}}{n!}x^ny^k,
	\end{align}
	and
	\begin{align*}
		\mathcal{R}_{*,\,k}(x):=\mathcal{R}(x,1)=\sum_{n\ge0}\frac{R_{n,\,k}}{n!}x^n,\text{ and }	
		\mathcal{R}_{n,\,*}(y):=\sum_{k\ge 0}R_{n,\,k}y^k.
	\end{align*}

In order to derive the explicit formula for $\mathcal{R}(x,y)$, we need to consider the recurrence relations of the generating functions of $R_{n,\,k}$ with $n$ or $k$ fixed.

	\begin{thm}\label{thm-R(*,k)-recurrence}
		For $k\ge 1$, we have
		\begin{align}\label{eq-rec-R(*,k)-int}
		\mathcal{R}_{*,\,k}(x)=(2k+1)\int_0^x \mathcal{R}_{*,\,k}(t) \,{\rm d} t+x\mathcal{R}_{*,\,k-1}(x)+(1-2k)\int_0^x \mathcal{R}_{*,\,k-1}(t)\,{\rm d}t,
		\end{align}	
		and
		\begin{align}\label{eq-R(*,0)}
			\mathcal{R}_{*,\,0}(x)=e^x.
		\end{align}
	\end{thm}

	\begin{proof}
		For $k=0$, according to Theorem \ref{thm-rec-R(n,k)}, we have $R_{n,\,0}=1$. Therefore,
		\[\mathcal{R}_{*,\,0}(x)=\sum_{n\ge 0}\frac{x^n}{n!}=e^x.\]	
		For $k\ge 1$, using Equation~\eqref{eq-rec-R(n,k)}, we have

		\begin{align*}
			\mathcal{R}_{*,\,k}(x)&=\sum_{n\ge 1}\frac{1}{n!}\left((2k+1) R_{n-1,\,k}+(n-2k+1)R_{n-1,\,k-1} \right)x^n \nonumber \\
			&=(2k+1)\sum_{n\ge 1}\frac{R_{n-1,\,k}}{n!}x^n+\sum_{n\ge1}\frac{R_{n-1,\,k-1}}{(n-1)!}x^n+(1-2k)\sum_{n\ge 1}\frac{R_{n-1,\,k-1}}{n!}x^n \nonumber \\
			&=(2k+1)\int_0^x \mathcal{R}_{*,\,k}(t) \,{\rm d} t+x\mathcal{R}_{*,\,k-1}(x)+(1-2k)\int_0^x \mathcal{R}_{*,\,k-1}(t)\,{\rm d}t.
		\end{align*}
			
		\vspace{-0.85cm}
		
	\end{proof}
	
	By applying the above theorem, it is straightforward to obtain the following partial differential equation for $\mathcal{R}(x,y)$.

	\begin{thm}\label{thm-R(x,y)-recurrence}
		We have
		\begin{align}\label{eq-R(x,y)-recurrence}
			\mathcal{R}(x,y)-2(y-1)y \frac{\partial \mathcal{R}(x,y)}{\partial y}+(xy-1)\frac{\partial \mathcal{R}(x,y)}{\partial x}=0.
		\end{align}
	\end{thm}
	
	\begin{proof}
		By the definition of $\mathcal{R}(x,y)$ and from Equations~\eqref{eq-rec-R(*,k)-int} and~\eqref{eq-R(*,0)}, we have
		\begin{align*}
			\mathcal{R}(x,y)=&\,e^x+\sum_{k\ge 1} \mathcal{R}_{*,\,k}(x)y^k\\
			=&\,e^x+\sum_{k\ge 1}\left(\int_0^x \mathcal{R}_{*,\,k}(t)\,{\rm d}t\right) y^k+\sum_{k\ge 1}2k\left(\int_0^x \mathcal{R}_{*,\,k}(t)\,{\rm d}t\right) y^k\\
			&+x\sum_{k\ge 1}\mathcal{R}_{*,\,k-1}(x)y^k+\sum_{k\ge 1}(1-2k)\left(\int_0^x \mathcal{R}_{*,\,k-1}(t)\,{\rm d}t\right) y^k\\
			=&\,e^x+\int_0^x \left(\sum_{k\ge1}\mathcal{R}_{*,\,k}(t)y^k\right){\rm d}t+2y\int_0^x \left(\sum_{k\ge 1}k\mathcal{R}_{*,\,k}(t)y^{k-1}\right) {\rm d}t+xy\sum_{k\ge 0}\mathcal{R}_{*,\,k}(x)y^k\\
			&-y\int_0^x \left(\sum_{k\ge 0}\mathcal{R}_{*,\,k}(t)y^k\right) {\rm d}t-2y^2\int_0^x \left(\sum_{k\ge 0}k\mathcal{R}_{*,\,k}(t)y^{k-1}\right){\rm d}t\\
			=&\,e^x+\left(\int_0^x\mathcal{R}(t,y)-e^x \right){\rm d}t+\int_0^x 2y\frac{\partial\mathcal{R}(t,y)}{\partial y}{\rm d}t+xy\mathcal{R}(x,y)\\
			&-y\int_0^x \mathcal{R}(t,y) \,{\rm d}t-2y^2\int_0^x \frac{\partial \mathcal{R}(t,y)}{\partial y}{\rm d}t\\
			=&\,(1-y)\left(2y\int_0^x \frac{\partial \mathcal{R}(t,y)}{\partial y}{\rm d}t+\int_0^x \mathcal{R}(t,y)\,{\rm d}t \right)+xy\mathcal{R}(x,y).
		\end{align*}
		By taking the partial derivative with respect to $y$ on both sides of the above equation, we obtain
		\[\frac{\partial \mathcal{R}(x,y)}{\partial x}=(1-y)\left(2y \frac{\partial \mathcal{R}(x,y)}{\partial y}+\mathcal{R}(x,y)\right)+y\mathcal{R}(x,y)+xy\frac{\partial\mathcal{R}(x,y)}{\partial x}, \]
		which is clearly equivalent to the desired equation.
	\end{proof}

	In fact, the partial differential equation~\eqref{eq-R(x,y)-recurrence} has a unique power series solution satisfying the given boundary conditions.
	
	\begin{lem}\label{lem-unique-sol}
		Let $\mathcal{F}_1(x)=\sum_{n\geq0}s_{n}x^n\in \mathbb{Q}[[x]]$ and $\mathcal{F}_2(y)=\sum_{k\geq0}s'_{k}y^k\in\mathbb{Q}[[y]]$ be two power series such that $s_0=s'_0$. Then the equation system
		\begin{equation}\label{eq-PDE}
		\left\{\begin{aligned}
		&\mathcal{F}(x,y)-2(y-1)y \frac{\partial \mathcal{F}(x,y)}{\partial y}+(xy-1)\frac{\partial \mathcal{F}(x,y)}{\partial x}=0,\\
		&\mathcal{F}(x,0)=\mathcal{F}_1(x),\quad\mathcal{F}(0,y)=\mathcal{F}_2(y)
		\end{aligned}\right.
		\end{equation}
		has a unique power series solution $\mathcal{F}(x,y)\in\mathbb{Q}[[x,y]]$.
	\end{lem}
	\begin{proof}
		Assume that $\mathcal{F}(x,y)=\sum_{n,\,k\geq0}a_{n,\,k}x^ny^k$ is a power series solution of the equation system~\eqref{eq-PDE}. Then substituting it into the equation system~\eqref{eq-PDE}, we have
		\begin{equation*}
		\left\{\begin{aligned}
		&\sum_{n,\,k\geq0}a_{n,\,k}x^ny^k-2(y-1)y \sum_{n,\,k\geq0}a_{n,\,k}kx^ny^{k-1}+(xy-1)\sum_{n,\,k\geq0}a_{n,\,k}nx^{n-1}y^k=0,\\
		&\sum_{n\geq0}a_{n,0}x^n=\sum_{n\geq0}s_{n}x^n,\quad\sum_{k\geq0}a_{0,k}y^k=\sum_{k\geq0}s'_{k}y^k.
		\end{aligned}\right.
		\end{equation*}
		Comparing the coefficients of the power series, we obtain the following equivalent equation with respect to the sequence $(a_{n,\,k})_{n,\,k\geq0}$:
		\begin{equation}\label{eq-PDE-re}
		\left\{\begin{aligned}
		&na_{n,\,k}=(2k+1)a_{n-1,k}+(n-2k+1)a_{n-1,k-1},\quad n,k\geq1,\\
		&a_{n,0}=s_n,\quad n\geq0\\
		&a_{0,k}=s'_k,\quad k\geq0.
		\end{aligned}\right.
		\end{equation}
It is clear that the system of equations~\eqref{eq-PDE-re} has a unique sequence solution, and thus the system~\eqref{eq-PDE} has a corresponding unique power series solution.	\end{proof}
	
  Finally, we can derive an explicit formula for $\mathcal{R}(x,y)$.
  This formula provides an alternative form of
   \[p(x,y)=\frac{\sqrt{1-y}}{\sqrt{1-y}\cosh(x\sqrt{1-y})-\sinh(x\sqrt{1-y})},\]
   which was presented in \cite{Zhuang2016} as the exponential generating function for $n$-permutations with $k$ increasing runs of length at least $2$.
	
	\begin{thm}\label{thm-R(x,y)}
		We have
		\begin{align}\label{eq-R(x,y)}
			\mathcal{R}(x,y)=\frac{\sqrt{y-1}}{\sqrt{y-1}\cos(x\sqrt{y-1})-\sin(x\sqrt{y-1})}.
		\end{align}
	\end{thm}
	\begin{proof}
        Combining Theorems \ref{thm-R(x,y)-recurrence} and \ref{lem-unique-sol}, we can give an explicit formula for $\mathcal{R}(x,y)$ by solving Equation~\eqref{thm-R(x,y)-recurrence} and verifying the boundary conditions $\mathcal{R}(x,0)=e^x$ and $\mathcal{R}(0,y)=1$ directly.
	\end{proof}

\subsubsection{Real-rootedness}\label{subsec-R-zeros}
Now we discuss the {\em real-rootedness} of $\mathcal{R}_{n,\,*}(y)$. Let us review some definitions and results on the {\em generalized Sturm sequence}. Given two real-rooted polynomials $f(x)$ and $g(x)$ with nonnegative real coefficients,
let $\{r_{i}\}$ and $\{s_{j}\}$ be all roots of $f(x)$ and $g(x)$ in nonincreasing order, respectively. We say that $g(x)$ {\it interlaces} $f(x)$,
denoted by $g(x)\preccurlyeq f(x)$, if either $\deg f(x)=\deg g(x)=n$ and
$$s_{n}\leq r_{n}\leq s_{n-1}\leq r_{n-1}\leq \cdots\leq s_{1}\leq r_{1},$$
or $\deg f(x)=\deg g(x)+1=n$ and
$$r_{n}\leq s_{n-1}\leq r_{n-1}\leq \cdots\leq s_{1}\leq r_{1}.$$
Following Liu and Wang \cite{Liu-Wang-2006}, let $a\preccurlyeq bx+c$ for any real numbers $a,b,c$, and let $f(x)\preccurlyeq 0$ and  $0\preccurlyeq f(x)$ for any real-rooted polynomial $f(x)$. Moreover, we say that $f(x)$ is {\em standard} if $f(x)=0$ or its leading coefficient is positive.
Let $\{f_{n}(x)\}_{n\geq 0}$ be a sequence of standard polynomials. If for all $n\geq 0$, we have $f_{n}(x)\preccurlyeq f_{n+1}(x)$ and $f_n(x)$ is real-rooted, then $\{f_{n}(x)\}_{n\geq 0}$ is said to be a {\it generalized Sturm sequence}. For surveys on this topic, we refer readers to Stanley~\cite{Stanley-log-1989}, Brenti \cite{Brenti-1994}, and Br\"and\'en \cite{Branden-2015}.

Liu and Wang \cite{Liu-Wang-2006} provided a unified criterion for determining whether a polynomial sequence is a generalized Sturm sequence, and it is useful to prove the real-rootedness of a polynomial.

\begin{lem}\label{thm-lw}
	{\upshape(\cite{Liu-Wang-2006}, Corollary 2.4)} Let $\{f_{n}(x)\}_{n\geq 0}$ be a sequence of polynomials with nonnegative  coefficients with $\deg f_{n}(x)=\deg f_{n-1}(x)$ or $\deg f_{n-1}(x)+1$, which satisfies the following conditions:
	\begin{itemize}
		\item[{\rm (a)}] $f_{0}(x)$ and $f_{1}(x)$ are real-rooted polynomials with $f_{0}(x)\preccurlyeq f_{1}(x)$.
		
		\item[{\rm (b)}] There exist polynomials $a_{n}(x), b_{n}(x), c_{n}(x)$ with real coefficients such that
		\[f_{n}(x)=a_{n}(x)f_{n-1}(x)+b_{n}(x)f'_{n-1}(x)+c_{n}(x)f_{n-2}(x).\]
	\end{itemize}
	If for all $x\leq 0$, we have $b_{n}(x)\leq 0$ and $c_{n}(x)\leq 0$, then $\{f_{n}(x)\}_{n\geq 0}$ forms a generalized Sturm sequence.
\end{lem}

\begin{thm}\label{thm-Rn(x)}
	The polynomial sequence $\{\mathcal{R}_{n,\,*}(y)\}_{n\ge 0}$ is a generalized Sturm sequence. In particular, the polynomials $\mathcal{R}_{n,\,*}(y)$ are real-rooted.
\end{thm}

\begin{proof}

Firstly, from Equation~\eqref{eq-rec-R(n,k)}, we can derive the following recurrence relation for $\mathcal{R}_{n,\,*}(y)$:
\begin{equation}\label{R{n,}(y)}
	\mathcal{R}_{n,\,*}(y) = \left((n-1)y + 1\right)\mathcal{R}_{n-1,\,*}(y) - 2y(y-1)\mathcal{R}_{n-1,}'(y), \quad \text{for } n \geq 1,
\end{equation}
and the initial conditions are $\mathcal{R}_{0,\,*}(y) = \mathcal{R}_{1,\,*}(y) = 1$. Now by Lemma~\ref{thm-lw}, we only need to show that $\deg \mathcal{R}_{n,\,*}(y) = \left\lfloor \frac{n-1}{2} \right\rfloor$ by verifying that $R_{n,\,k}>0$ if and only if $0\leq k \leq \lceil \frac{n-1}{2}\rceil$.

We proceed by induction on $n\ge 1$. The base case is trivial since $\mathcal{R}_{1,\,*}(y)=1$.
For $n>1$, assume that $R_{n-1,\,k}>0$ if and only if $0\leq k \leq \lceil \frac{n-2}{2}\rceil$. By Equation~\eqref{eq-rec-R(n,k)}, we have $R_{n,\,k}\geq (n-2k+1)R_{n-1,\,k-1}>0$ for $0\leq k\leq  \lceil \frac{n-1}{2}\rceil$. If $k>\lceil \frac{n-1}{2}\rceil+1$, then $k>k-1>\lceil \frac{n-1}{2}\rceil\geq\lceil \frac{n-2}{2}\rceil$, so $R_{n-1,\,k}=R_{n-1,\,k-1}=0$, implying that $R_{n,\,k}=0$. Now, assume that $k=\lceil \frac{n-1}{2}\rceil+1$. Then $R_{n-1,\,k}=0$. If $n$ is odd, then $k=\frac{n-1}{2}+1$, and thus $n-2k+1=0$. If $n$ is even, then $k=\frac{n}{2}+1$, so $k-1=\frac{n}{2}>\frac{n-2}{2}=\lfloor \frac{n-2}{2}\rfloor$, implying that $R_{n-1,\,k-1}=0$. Therefore, $R_{n,\,\lceil \frac{n-1}{2}\rceil+1}=0$, and this completes the proof.
\end{proof}
Note that Ma \cite{MA2012} also presented Equation \eqref{R{n,}(y)} and obtained the real-rootedness of the relevant polynomials.

\subsection{Weakly increasing permutations in $S_n^3$}\label{WI-perms}
This subsection is devoted to presenting some enumeration results for weakly increasing permutations in $S_n^3$. Notably, we present two proofs to show that these permutations can be counted by the Springer numbers.

We begin by providing the relevant definitions.
A {\em weakly increasing permutation} is defined as a permutation that avoids the pattern $21$  (equivalently, as a permutation that avoids the pattern  $\underline{21}$).
Let
$$\WI_n:=|S^3_n(21)|=|\{\Pi\in S^3_n :\Pi \mbox{ is weakly increasing}\}|.$$
We next present the following theorem, which establishes a connection between weakly increasing permutations and permutations in $S_n^3$ with $k$ levels repeated.

\begin{thm}\label{WI-rec-relations}
	For $n\geq0$, we have
	\begin{equation}\label{relation-WInk-Rnk}
		\WI_{n}=\sum_{k\ge 0}2^k R_{n,\,k}.
	\end{equation}
\end{thm}
\begin{proof}
	It is clear that any weakly increasing permutation can be obtained from a canonical permutation counted by $R_{n,\,k}$ by fixing the order of elements at the same level in $2^k$ ways and then arranging all elements in weakly increasing order. Summing over all possible $k$, we obtain Equation~(\ref{relation-WInk-Rnk}), as desired.
\end{proof}

	Let
	\[\mathcal{W}(x):=\sum_{n\ge 0}\frac{\WI_n}{n!}x^n\]
	be the exponential generating function for weakly increasing permutations. We can derive an explicit formula for $\mathcal{W}(x)$ as follows.
	
	\begin{thm}\label{thm-W}
		We have
		\begin{align}\label{eq-W}
			\mathcal{W}(x)=\frac{1}{\cos(x)-\sin(x)}.
		\end{align}
	\end{thm}

	\begin{proof}
		By Equation~\eqref{relation-WInk-Rnk}, we have
		\[\mathcal{W}(x)=\sum_{n\ge 0}\sum_{k\geq 0}\frac{2^kR_{n,\,k}}{n!}x^n=\mathcal{R}(x,2),\]
		which yields Equation~\eqref{eq-W} by using Equation~\eqref{eq-R(x,y)}.
	\end{proof}

The exponential generating function~\eqref{eq-W} of the numbers $\WI_n$ corresponds to that for sequence A001586 in~\cite{oeis}. This shows that weakly increasing permutations are counted by the {\em Springer numbers}~$\mathcal{S}_n$~\cite{Gla1898,Gla1914,Springer1971}, thus providing another combinatorial interpretation for these numbers.

Most recently, Han et al.~\cite{HKZ2024} proved a conjecture by Callan that the number of up-down permutations of even length fixed under reverse and complement are given by the Springer numbers. Examples of such permutations are 1302, 251403 and 46570213. This combinatorial interpretation can be seen as one of the easiest to define. The approach taken to prove Callan's conjecture in \cite{HKZ2024} involved deriving the following relationship between the Springer numbers and the {\em Euler numbers}:
\begin{equation}\label{Springer-Euler}
	\mathcal{S}_n=\sum_{k=0}^{n-1}{n-1\choose k}2^kE_k\mathcal{S}_{n-k-1},
\end{equation}
where the Euler numbers $E_n$ are defined by the exponential generating function
\begin{equation}\label{E_n}
	E(x):=\sum_{n\ge 0}E_n\frac{x^n}{n!}=\sec x+\tan x.
\end{equation}

In the rest of this subsection, we provide another proof of Theorem~\ref{thm-W} by deriving the recurrence relation~(\ref{eq-rec-WI}), similar to the relation~(\ref{Springer-Euler}). For this purpose, we introduce the following definitions. Given a permutation $\pi=\pi_1\pi_2\cdots \pi_n \in S_n^2$, a {\it peak} in $\pi$ is an index $i$ such that $\pi_{i-1}<\pi_i$ and $\pi_{i+1}<\pi_i$, for $2\le i \le n-1$. Let
\[P_{n,\,k}=|\{\pi\in S_n^2 : \pi \text{ has exactly $k$ peaks}\}|,  \]
and let
\[\mathcal{P}(x,y):=\sum_{n,\,k\ge 0}\frac{P_{n,\,k}}{n!}x^ny^k. \]
After simplifying the original expression for $\mathcal{P}(x,y)$ in \cite[Corollary 23]{Sergey-2006}, we obtain the following lemma.
\begin{lem}\label{lem-gf-peak}{\upshape{(\cite{Sergey-2006})}}
We have
\begin{align}\label{eq-P(x,y)}
\mathcal{P}(x,y)=\frac{\sqrt{y-1}}{\sqrt{y-1}-\tan(x\sqrt{y-1})}.
\end{align}
\end{lem}

By applying the above lemma, we derive the following recurrence relation for the numbers $\WI_n$. Considering the relation~(\ref{Springer-Euler}), this provides an alternative proof of Theorem~\ref{thm-W}.

\begin{thm}
For $n\ge 1$, we have
\begin{equation}\label{eq-rec-WI}
\WI_n=\sum_{k=0}^{n-1}\binom{n-1}{k}2^kE_k\WI_{n-k-1}.
\end{equation}
\end{thm}
\begin{proof}
	
	For each permutation $\Pi \in S_n^3(21)$, we can decompose it uniquely into two fragments. One fragment is of the following form:
	\begin{align}\label{eq-collection-a_i}
		(n-1,a_1)^T,\,(a_1,a_2)^T,\,(a_2,a_3)^T,\ldots,(a_{k-1},a_k)^T,\,(a_k,n-1)^T,
	\end{align}
	where the $a_i$'s are distinct numbers. If $k=0$, then this fragment consists of exactly one element: $(n-1,n-1)^T$.
	The other fragment consists of the remaining elements of $\Pi$.
	Both fragments are in the same order as they appear in the  original permutation $\Pi$. Note that the decomposition is unique and reversible since no two elements on the same level can be in different fragments. We need to prove the following claim:

\noindent
{\bf Claim:}
	Given a positive integer $k$, the number of possible
	fragments of the form \eqref{eq-collection-a_i} is
	\[\sum_{\ell=0}^k P_{k,\,\ell}2^{\ell+1}. \]

	{\noindent \bf Proof of Claim.} We identify the fragment in \eqref{eq-collection-a_i} with a permutation $\pi=a_1a_2\cdots a_k \in S_k^2$. Then, for $2\le i\le k-1$, the collection in \eqref{eq-collection-a_i} contains a pair on level $a_i$ if and only if $a_{i-1}<a_i$ and $a_{i+1}<a_i$, which indicates that $i$ is a peak in $\pi$.
	Thus, if $\pi$ has $\ell$ peaks, then the fragment in \eqref{eq-collection-a_i} contains exactly $\ell+1$ pairs of elements at the same level. Moreover, there are $2^{\ell +1}$ ways to determine the internal sorting of these elements. Therefore, the number of ways to form a fragment of the form \eqref{eq-collection-a_i} is $\sum_{\ell=0}^kP_{k,\,\ell}2^{\ell+1}$, as desired.

In the recurrence relation~\eqref{eq-rec-WI}, $k$ denotes the number of $a_i$'s in the first fragment, and those entries $a_1,\,a_2,\ldots,a_k$ can be chosen in $\binom{n-1}{k}$ ways. For the second fragment, there are $\WI_{n-k-1}$ choices,  since we can first choose a weakly increasing permutation of length $n-k-1$ and then replace $i$ in this permutation, $0\leq i\leq n-k-2$, with the $i$-th smallest element in $\{0,1,\ldots,n-2\}\backslash\{a_1,\ldots,a_k\}$. Thus, it remains to show that for $k\geq1$,
	\[\sum_{\ell=0}^k P_{k,\,\ell}2^{\ell+1}=2^k E_k, \]
	by using the above claim.

	On the one hand, by substituting $t=2x$ in Equation~\eqref{E_n}, we obtain that
	\begin{align*}
		\sum_{k\ge 1}E_k2^k\frac{x^k}{k!}&=\tan(2x)+\sec(2x)-1\\[-5pt]
        &=\frac{2\tan(x)}{1-\tan^2(x)}+\frac{1-\cos^2(x)+\sin^2(x)}{\cos^2(x)-\sin^2(x)}\\
		&=\frac{2\tan(x)}{1-\tan^2(x)}+\frac{2\tan^2(x)}{1-\tan^2(x)}\\
		&=\frac{2\tan(x)}{1-\tan(x)}.
	\end{align*}

On the other hand, by applying Lemma \ref{lem-gf-peak}, we get
	\begin{align*}
       \sum_{k\ge 1}\left(\sum_{\ell=0}^kP_{k,\,\ell}2^{\ell +1}\right)\frac{x^k}{k!}=2(\mathcal{P}(x,2)-1)=\frac{2}{1-\tan(x)}-2=\frac{2\tan(x)}{1-\tan(x)},
	\end{align*}
	which coincides with the exponential generating function for $E_k2^k$ and completes the proof.
\end{proof}

\subsection{Unimodal permutations in $S_n^3$}\label{unimodal-sec}

This subsection is devoted to the study of the total number of unimodal permutations in $S_n^3$. More preciously, we establish a connection between this number and $R_{n,\,k}$, and further derive the precise formula for its exponential generating function using the results in Subsection~\ref{k-lev-rep-sec}.

A {\em unimodal permutation} is a permutation whose vector of levels begins with several (possibly none) ascents, followed by several (possibly none) descents. Examples of unimodal permutations include:
	$$\left(\begin{array}{ccccc}3&4&2&0&1\\ 1&4&3&2&0\\ \end{array}\right),\  \left(\begin{array}{ccccc}0&1&2&3&4\\ 0&1&2&3&4\\ \end{array}\right)
	\mbox{ and }\left(\begin{array}{ccccccc}1&4&5&6&3&2&0\\ 0&2&3&6&5&4&1\\ \end{array}\right).$$
	
Let
	\[ U_n=|\{\Pi\in S_n^3 : \Pi \text{ is unimodal} \}|,\]
	and set $U_0=0$. Then $U_n$ can be expressed in terms of the numbers $R(n,k)$, as stated in the next theorem.
	\begin{thm}
		For $n\geq 1$, we have
		\begin{equation}\label{eq-U(n)}
			U_n=\sum_{k\geq0} 2^{n-k-1}R_{n-1,\,k}.
		\end{equation}
\end{thm}

\begin{proof}
Note that any unimodal permutation in $S_n^3$ must include the element $\Pi_i=(n-1,n-1)^T$, for some $i$, $1\le i\le n$. As in the proof of Theorem \ref{thm-rec-R(n,k)}, we treat two elements at the same level as a pair. To determine the number of unimodal permutations $\Pi$ in $S_n^3$, we count the possible arrangements of $n-1$ elements to the left and to the right of $\Pi_i$.

	Suppose that the number of pairs of elements at the same level is $k$, with $k\ge 0$.
	Then, we need to arrange the two elements of each pair on opposite sides of the element $\Pi_i$, providing us with $2^{k}$ choices. Following this, we need to decide how to arrange the remaining $n-1-2k$ elements, which gives us $2^{n-1-2k}$ choices for their placements. In this way, to obtain $n-1$ elements with exactly $k$ pairs, we first select a set of elements with $k$ pairs that can form a permutation $\Pi\in S_{n-1}^3$ in $R_{n-1,\,k}$ ways. Thus, the total count becomes $2^{n-k-1} R_{n-1,\,k}$.
	
	Summing over all possible $k$, we obtain Equation~\eqref{eq-U(n)}.
\end{proof}

	Let
	\[U(x):= \sum_{n\ge0}\frac{U_n}{n!}x^n.\]
Then we can derive an explicit formula for $U(x)$ by applying the above theorem together with Theorem \ref{thm-R(x,y)}.
	
	\begin{thm}\label{thm-U(x)}
		We have
		\begin{align}\label{eq-U(x)}
			U(x)=\frac{1}{\sqrt{2}}\left({\rm arcsinh}\left(\frac{1}{\sinh\left({\rm arcsinh}(1)-\sqrt{2}x\right)}\right) -{\rm arcsinh}(1)\right).
		\end{align}
	\end{thm}
	
	\begin{proof}
		Using Equation~(\ref{eq-U(n)}), we have
		\begin{align*}
			U(x)&=\sum_{n\ge 1}\frac{1}{n!}\sum_{k\ge0}2^{n-k-1}R_{n-1,\,k}x^n\\[5pt]
			&=\sum_{n\ge 1}\sum_{k\ge 0}\frac{1}{2 n!}R_{n-1,\,k}(2x)^n\left( \frac{1}{2}\right)^k\\[5pt]
			&=\int_0^x \mathcal{R}\left(2t,\frac{1}{2}\right){\rm d}t.
		\end{align*}
		Using the expression for $\mathcal{R}(x,y)$ in Equation~\eqref{eq-R(x,y)}, we have
		\begin{align*}
			U(x)&=\int_0^x \frac{i}{i\cos (\sqrt{2}t i)-\sqrt{2}\sin(\sqrt{2}t i)}{\rm d}t\\[5pt]
			&=\int_0^x \frac{1}{\cos (\sqrt{2}t i)+\sqrt{2}\,i\sin(\sqrt{2}t i)}{\rm d}t\\[5pt]
			&=\int_0^x \frac{1}{\cosh (\sqrt{2}t )-\sqrt{2}\sinh(\sqrt{2}t)}{\rm d}t\\[5pt]
			&=\int_0^x \frac{1}{\sinh({\rm arcsinh}(1)-\sqrt{2}t)}{\rm d}t,
		\end{align*}
		where the third and fourth equalities hold since
		\[\cos(it)=\cosh(t),\quad \sin(it)=i\sinh(t), \text{ and }\sinh(t+s)=\sinh(t)\cosh(s)+\cosh(t)\sinh(s).\]
		Let $T={\rm arcsinh}(1)-\sqrt{2}t$, then we have
		\begin{align*}
			U(x)&=-\frac{1}{\sqrt{2}}\int_0^x \frac{\cosh(T)}{\sinh^2(T)\sqrt{\cosh^2(T)/\sinh^2(T)}}{\rm d}\,T\\
			&=-\frac{1}{\sqrt{2}}\int_0^x \frac{\cosh(T)}{\sinh^2(T)\sqrt{1/\sinh^2(T)+1}}{\rm d}\,T.
		\end{align*}
		On the other hand, if we let
		\[F(x)=\frac{1}{\sqrt{2}}\left({\rm arcsinh}\left(\frac{1}{\sinh\left({\rm arcsinh}(1)-\sqrt{2}x\right)}\right) -{\rm arcsinh}(1)\right), \]
		then by differentiating $F(x)$, we obtain
		\[F'(x)=\frac{\cosh\left({\rm arcsin}(1)-\sqrt{2}x\right)}{\sinh^2\left({\rm arcsin}(1)-\sqrt{2}x\right)\sqrt{1/\sinh^2\left({\rm arcsin}(1)-\sqrt{2}x\right)+1}}.\]
		From the above derivation and verification of the initial values (i.e., $F(0)=U(0)$), we conclude that $U(x)=F(x)$. This completes the proof.
	\end{proof}
	By Theorem \ref{thm-U(x)}, we can deduce that the numbers $U_n$ correspond to sequence A104018 in \cite{oeis}, thereby providing the first combinatorial interpretation for this sequence.

\subsection{Hoe permutations in $S_n^d$}\label{hoe-section}
In this subsection we consider a subclass of $d$-dimensional unimodal permutations, which we call {\em hoe permutations}.

We define $\Pi \in S_n^d$ as a \emph{hoe permutation} if $\Pi$ is unimodal and has exactly one descent at position $n-1$. For instance, the permutations
\[0124563\in S^2_7,\ \left(\begin{array}{ccccc}1&0&2&4&3\\ 0&2&3&4&1\\ \end{array}\right)\in S_5^3\mbox{ and }\left(\begin{array}{cccc}1&0&3&2\\ 0&2&3&1\\1&2&3&0 \end{array}\right)\in S_4^4
\]
are all hoe permutations. Note that any such $\Pi$ must contain the element $(n-1,\ldots,n-1)^T$.
Let
$$H_{n,\,d}= |\{ \Pi\in S_n^d : \Pi \text{ is a hoe permutation} \}|.$$
We can then derive the following explicit formula for $H_{n,\, d}$ in a combinatorial way.
\begin{thm}
	For $n\geq 2$ and $d\geq 2$, we have
	\begin{equation}\label{eq-H(n,d)}
		H_{n,\,d}=\sum_{k=0}^{n-2}\left(2^{d-1}-1 \right)^k .
	\end{equation}
\end{thm}

\begin{proof}
	We first consider the case of $d=2$, where a $2$-dimensional hoe permutation of length $n$ is just an ordinary permutation with exactly one descent at position $n-1$. There are $n-1$ choices to obtain this, coinciding with Equation~\eqref{eq-H(n,d)}.
	
	Next, for $d \ge 3$, we claim that the numbers $H_{n,\,d}$ satisfy the following recurrence relation:
	\begin{equation}\label{eq-recurrence-H(n,d)}
		H_{n,\,d}=1+(2^{d-1}-1)H_{n-1,\,d}.
	\end{equation}
	In fact, to obtain a hoe permutation $\Pi$ in $S_n^d$, we can insert the element $C=(n-1,\ldots,n-1)^T$ between the two rightmost elements of a permutation $\Pi'$ in $S_{n-1}^d$, where the first $n-2$ elements are strictly increasing. It is now sufficient to count the number of such $\Pi'$'s in $S_{n-1}^d$. Note that there are three cases for the patterns formed by the two rightmost elements in $\Pi'$'s. If they form an ascent, then $\Pi'$ is strictly increasing, which counts as $1$. If they form a descent, then $\Pi'$ is a hoe permutation, counted by $H_{n-1,\,d}$. Otherwise, these elements must form a plateau. Since all occurrences of ``$n-2$" must be in the two rightmost elements of $\Pi'$, such a permutation can be obtained from a hoe permutation by freely swapping the two entries of the two rightmost elements within the same rows (excluding the cases of swapping the entries in all rows or not swapping at all). Thus there are $(2^{d-1}-2)H_{n-1,\,d}$ such $\Pi'$'s in this case. Combining these three cases, we obtain Equation~\eqref{eq-recurrence-H(n,d)}.

	Rewriting Equation~\eqref{eq-recurrence-H(n,d)} as
	\begin{equation}\label{eq-recurrence-H(n,d)-2}
		H_{n,\,d}+\frac{1}{2^{d-1}-2}=(2^{d-1}-1)\left( H_{n-1,\,d}+\frac{1}{2^{d-1}-2}\right),
	\end{equation}
	then $\{H_{n,\,d}+1/(2^{d-1}-2)\}_{n\geq2}$ is a geometric progression with the initial value $H_{2,d}+1/(2^{d-1}-2)=1+1/(2^{d-1}-2)$. After simple computation, we have
	\begin{align*}
		H_{n,\,d}
		=\frac{(2^{d-1}-1)^{n-1}-1}{\left(2^{d-1}-1 \right)-1},
	\end{align*}
	which coincides with Equation~\eqref{eq-H(n,d)}, and this completes the proof.
\end{proof}

The following special cases match sequences in the OEIS \cite{oeis}:
\begin{equation}
	\begin{matrix}
		d=3: & H_{n,3}=\frac{3^{n-1}-1}{2}   & \cite[A003462]{oeis} \\[6pt]
		d=4: & H_{n,4}=\frac{7^{n-1}-1}{6}   & \cite[A023000]{oeis} \\[6pt]
		d=5: & H_{n,5}=\frac{15^{n-1}-1}{14} & \cite[A135518]{oeis} \\[6pt]
		d=6: & H_{n,6}=\frac{31^{n-1}-1}{30} & \cite[A218734]{oeis}
	\end{matrix}		
\end{equation}
Notably, for $d=4,5,6$, we provide the first known combinatorial interpretations for the respective sequences. Additionally, one of the known combinatorial interpretations in A003462 is related to the number of triangles in {\em Sierpinski's triangle}, while in Subsection \ref{c-bounded-perms-sec}, we will introduce the numbers $|M^{+3}_{3,n}|$, which is linked in \cite[A002023]{oeis} to Sierpinski's tetrahedron (i.e., to the number of edges in the $(n-2)$-Sierpinski tetrahedron graph).

\subsection{$(d-1)$-plateaux among permutations in $S_n^d$}\label{sec-plaeatu-levMax}

In various parts of the permutation patterns literature, researchers are interested in the total number of occurrences of a specific pattern in all permutations (often under certain restrictions) in  $S_n$. For example, see \cite{BurEli, FriMan} and references therein. In this subsection, we derive the explicit formula for the total number of $(d-1)$-plateaux among all permutations in $S_n^d$.

\begin{thm}\label{distr-k-plateau-all-2}
	Let $n,\,d\geq2$. If $n\geq d-1$, then the total number of $(d-1)$-plateaux (occurrences of $\underline{11\ldots1}$ with $d-1$ $1$'s) among all permutations in $S^d_n$ is given by
	\begin{equation}\label{total-occ-k-plateau-1}
		(n-d+2)(d-1)!\left((n-d+1)!\right)^{d-1}\sum_{\ell=0}^{n-1} \left(A(\ell,d-2) \right)^{d-1},
	\end{equation}
	where $A(m,k)$ represents the {\em arrangement number}: $$A(m,k)=\left\{\begin{array}{ll}
	\frac{m!}{(m-k)!},& \mbox{if }0\leq k\leq m;\\
	0, & \text{otherwise.}\\
	\end{array}\right.$$
\end{thm}
\begin{proof}
	Suppose that the elements of a $(d-1)$-plateau are on level $\ell$ for some $0\le \ell \le n-1$. It is clear that for this $(d-1)$-plateau, there is exactly one $\ell$ in each row, and there are $(d-1)!$ choices to decide the internal ordering of these $\ell$'s.
	After that, we need to choose the remaining $d-2$ entries from the set $\{ 0,1,\ldots, \ell -1\}$ and this contributes to $(A(\ell,d-2))^{d-1}$. Summing over all possible $\ell$, we can obtain that there are $(d-1)!\sum_{\ell=0}^{n-1} \left(A(\ell,d-2) \right)^{d-1}$ choices for such a $(d-1)$-plateau.
	
	To complete the construction of a $(d,n)$-permutation, we still need to arrange the remaining elements with $\left((n-(d-1))! \right)^{d-1}$ choices, and then insert the $(d-1)$-plateau into these elements in $(n-(d-1)+1)$ ways. Consequently, we obtain Equation~\eqref{total-occ-k-plateau-1}.
\end{proof}
In the three-dimensional case, it is straightforward to calculate the total number of plateaux (occurrences of $\underline{11}$) among all permutations. Moreover, we can obtain the total number of ascents and descents in this case.
\begin{cor}\label{cor-plateau-S_n^3-levMax}
	Let $n\geq2$. The total number of plateaux among all permutations in $S^3_n$ is given by
	\begin{equation}\label{total-occ-plateau-2}
		\frac{(2n-1)n!(n-1)!}{3}.
	\end{equation}
	The total number of ascents (the same as descents) in $S^3_n$ is given by
	\begin{equation}\label{total-occ-asc} \frac{(3n^2-5n+1)n!(n-1)!}{6}.\end{equation}
\end{cor}

\begin{proof}
	Set $d=3$ in Equation~\eqref{total-occ-k-plateau-1} to obtain
	\[ (n-1)((n-2)!)^2 \sum_{\ell=0}^{n-1}2!\ell^2,\]
	which coincides with Equation~\eqref{total-occ-plateau-2} by substituting
	\[\sum_{\ell=1}^{n-1}\ell^2=\frac{(n-1)n(2n-1)}{6}.\]
	For the second assertion, considering the reversal of all permutations in $S^3_n$, we observe that the total number of ascents equals the total number of descents. Therefore, twice the total number of ascents plus the total number of plateau equals $(n-1)(n!)^{2}$, which represents the total number of consecutive pairs of elements in $S^3_n$.
	Thus, the total number of ascents is given by
	$$\frac{1}{2}\left((n-1)(n!)^{2}-\frac{(2n-1)n!(n-1)!}{3}\right).$$
	This completes the proof.
\end{proof}


\section{$\levSum$: levels defined by sums}\label{levSum-sec}

In this section, we introduce the definitions of minimal, maximal, minimax, and $c$-bounded permutations. These concepts are mostly trivial under $\levMax$, while they lead to several interesting results under $\levSum$. The common combinatorial objects in these two definitions of levels are the occurrences of patterns across all permutations.

\subsection{Minimal permutations}\label{min-perm-sec}
{In this subsection, we examine the minimal permutations and explicitly determine their maximum levels by presenting a construction of specific examples.}

We first provide the definition of minimal permutations for any level definition. For any $d\geq 2$, we let $m_{d,n}=\min\{\level_{\max}(\Pi):\Pi\in S_n^d\}$ and $M_{d,n}=\{\Pi\in S_n^d:\level_{\max}(\Pi)=m_{d,n}\}$. The permutations in $M_{d,n}$ are called {\em minimal permutations} or {\em minimal $(d,n)$-permutations} when we want to emphases the dimension and length of minimal permutations.

The notion of a minimal permutation is trivial for $\levMax$, as in this case $M_{d,n}=S^d_n$. Moreover, the notion of a minimal permutation is trivial for $\levSum$ when $d=2$ because $m_{2,n}=n-1$, meaning that each permutation is minimal. Therefore, $M_{2,n}=S_n^2=S_n$ in this case, with $n!$ such permutations. On the other hand, for $\levSum$ and $d=3$, it is easy to see that each minimal $(3,n)$-permutation must be a permutation of the columns of
\begin{equation}\label{minimal-sol-d=3} X_{3,n}=\left(\begin{array}{cccc} 0 & 1 & \dots & n-1 \\ n-1 & n-2 & \dots & 0  \end{array}\right)\end{equation}
and hence the number of minimal $(3,n)$-permutations in this case is again $n!$. The situation becomes more complicated for $\levSum$ when $d\geq 4$. We are able to determine $m_{d,n}$ for $\levSum$ in the general case by constructing classes of minimal permutations, but we have been unable to find the exact cardinality of $M_{d,n}$ in the general case.

{A special case is that} the size of $\Ca(M_{4,2k+1})$ corresponds to sequence A002047 in \cite{oeis}, which has several combinatorial interpretations, including essentially our minimal permutations (by subtracting $k$ from each entry in our permutation, one obtains an object in \cite[A002047]{oeis}). No formula is listed in \cite[A002047]{oeis}, suggesting that enumerating minimal permutations might be challenging.

In the case of $d\geq 3$, we derive the explicit formula for $m_{d,n}$ under $\levSum$, as stated in the following theorem.

\begin{thm}\label{minimal-thm}
	We have $m_{d,n}=\left\lceil \frac{(d-1)(n-1)}{2} \right\rceil$ for all $n\geq 1$ and $d\geq 3$.
\end{thm}

\begin{proof}
	We begin by finding a lower bound for $m_{d,n}$. Note that all entries in each row of a permutation $\Pi\in M_{d,n}$ sum up to $\frac{n(n-1)}{2}$, hence the total sum of all entries in $\Pi$ is
	\[\frac{n(n-1)(d-1)}{2} = \levSum(\Pi_1)+\levSum(\Pi_2)+\cdots+\levSum(\Pi_n),\]
	which is less than or equal to $n\cdot m_{d,n}$.
	Thus, we have
	\begin{equation}\label{lower-bound-on-mdn} m_{d,n}\geq  \frac{(d-1)(n-1)}{2}. \end{equation}
	
	Next it suffices to show that there exists a $(d,n)$-permutation $\Pi$ such that $\levSum_{\max}(\Pi)=\left\lceil \frac{(d-1)(n-1)}{2} \right\rceil$ for all $n\geq 1$ and $d\geq 3$.
	
	We have already presented the permutation $X_{3,n}$ in (\ref{minimal-sol-d=3}), whose elements are all at the same level $\left\lceil \frac{(d-1)(n-1)}{2} \right\rceil$ and the case of $n=1$ is trivial, so we assume that $d\geq 4$ and $n\geq2$. By letting the second row be a cyclic shift of the first row to the left by $k$ entries for $k \geq 1$, and filling in the third row such that the sum of each column becomes $3k$, we obtain a proof of the claim that $m_{4, 2k+1} = \left\lceil \frac{(4-1)((2k+1)-1)}{2} \right\rceil = 3k$:
$$Y_{4,2k+1}=\left(\begin{array}{ccccccccc} 0 & 1 & \dots & k-1 & k & k+1 & \dots & 2k-1 & 2k \\ k & k+1 & \dots & 2k-1 & 2k & 0 & \dots & k-2 & k-1 \\ 2k & 2k-2 & \dots & 2 & 0 & 2k-1 & \dots & 3 & 1  \end{array}\right).$$
Using a similar construction as for $Y_{4,2k+1}$, we prove the claim that $m_{4,2k}=\left\lceil \frac{(4-1)(2k-1)}{2} \right\rceil=3k-1$:
$$Y_{4,2k}=\left(\begin{array}{ccccccccc} 0 & 1 & \dots & k-1 & k & k+1 & \dots & 2k-2 & 2k-1 \\ k-1 & k & \dots & 2k-2 & 2k-1 & 0 & \dots & k-3 & k-2 \\ 2k-1 & 2k-3 & \dots & 1 & 0 & 2k-2 & \dots & 4 & 2  \end{array}\right).$$
Now, if $d=2s+1$ for $s\geq 1$, then we can take $s$ copies of  $X_{3,n}$ from (\ref{minimal-sol-d=3}) and stack them on top of each other to obtain the desired result. On the other hand, if $d=2s$ for $s\geq 2$, and $n=2k+1$ (resp., $n=2k$), we take $Y_{4,2k+1}$ (resp., $Y_{4,2k}$) and $s-2$ copies of $X_{3,2k+1}$ (resp., $X_{3,2k}$)  and stack them on top of each other to obtain the desired result.
\end{proof}

\begin{rem}\label{minimal-perms-many}  Note that by permuting rows and columns in our constructions in Theorem~\ref{minimal-thm}, one can obtain many more permutations in $M_{d,n}$. Additionally, more such permutations  in $M_{d,2k}$ can be obtained by swapping $2k-i$ with $2k-i+1$ in the last row of $Y_{4,2k}$ for $i=3,5,\ldots,2k-1$, and then permuting rows and columns of these permutations.
\end{rem}

The set ${\rm CP}_{d,n}$ of {\em complete-plateau permutations} is defined as
$$\{\Pi\in S_n^d:\level(\Pi_i)=\ell \mbox{ for }1\leq i\leq n \mbox{ and fixed }\ell\}=\{\Pi\in S_n^d:\underline{11}(\pi)=n-1\}.$$
Under $\levSum$, using our arguments in the proof of Theorem~\ref{minimal-thm}, we deduce that $\frac{n(n-1)(d-1)}{2}=n\ell$. Therefore, $$\ell=\frac{(n-1)(d-1)}{2}$$ is uniquely determined. Since $\ell$ must be an integer, at least one of
$n$ or $d$ must be odd. Moreover, we have $\levSum(\Pi_i)=\ell=m_{d,n}$ for all $1\leq i \leq n$, which implies that ${\rm CP}_{d,n}\subseteq M_{d,n}$. In fact, the following corollary establishes that \[{\rm CP}_{d,n}=\left\{\begin{array}{ll}
M_{d,n},& \text{if either $n$ or $d$ is odd;}\\
\varnothing, & \text{otherwise.}\\
\end{array}\right.\]
\begin{cor}\label{iff_min_per}
	Let $n\geq1$, $d\geq 2$ and either $n$ or $d$ is odd. A $(d,n)$-permutation $\Pi$ is in $M_{d,n}$ if and only if $\Pi$ is a complete-plateau permutation with all levels of its elements being $m_{d,n}$.
\end{cor}
\begin{proof}
	The sufficiency is clear, and we assume that $\Pi\in M_{d,n}$ for the necessity.
	
	Because either $n$ or $d$ is odd and $\Pi\in M_{d,n}$, we have $\levSum(\Pi_i)\leq m_{d,n}=\left\lceil \frac{(d-1)(n-1)}{2} \right\rceil=\frac{(d-1)(n-1)}{2}$ for $1\leq i \leq n$. If there exists some element, say $\Pi_{i_0}$, such that $\levSum(\Pi_{i_0})<\frac{(d-1)(n-1)}{2}$, then the total sum of all entries in $\Pi$ is $\frac{n(n-1)(d-1)}{2}=\levSum(\Pi_1)+\cdots+\levSum(\Pi_{i_0})+\cdots+\levSum(\Pi_n)<n\cdot \frac{(d-1)(n-1)}{2}$, a contradiction.
\end{proof}

{In general, it is challenging to construct all minimal permutations, even under the complete-plateau condition. Here, we only give some sufficient and necessary conditions for minimal $(4,2k+1)$-permutations in shift forms.}

\begin{thm}\label{shiftform}
	Let $n\geq1$, $d\geq 2$ and either $n$ or $d$ is odd. Given $s_2,s_3,\ldots,s_{d}\in\{1,2,\ldots,n-1\}$ and $\pi_{n}^2,\pi_{n}^3,\ldots,\pi_{n}^d\in\{0,1,\ldots,n-1\}$ such that $\gcd(s_i,n)=1$ for all $i$, a $(d-1)\times n$ array of numbers can be defined by \[\Pi=\left\{\pi^i_j\right\}_{\begin{subarray}{l} 2 \leq i\leq d \\ 1\le j\le n\end{subarray}},\]
where $ \pi_j^i\equiv\pi_n^i+j\cdot s_i \text{ $(\md\,n)$}$. Then $\Pi$ is a $(d,n)$-permutation (omitting the first row).
	In particular, let $d=4$ and $n=2k+1$ with $k\geq1$. Assume (w.l.o.g.) that $s_2=1$ and $\pi_n^2=n-1$.
	\begin{enumerate}
		\item[{\rm (a)}] If $s_3=1$, then $\Pi\in M_{4,n}$ if and only if $s_4=n-2$, $(\pi_n^3,\pi_n^4)=(k-1,1)$ or $(k,0)$.
		\item[{\rm (b)}] If $s_3=k$, then $\Pi\in M_{4,n}$ if and only if $s_4=k$, $(\pi_n^3,\pi_n^4)=(0,k)$ or $(k,0)$.
	\end{enumerate}		
\end{thm}
\begin{proof}
	We first show that the array $\Pi$ defined above is a $(d,n)$-permutation.
	Since $\gcd(s_i,n)=1$, there exists a $j\equiv s_i^{-1}(x-\pi^i_n)\text{ $(\md\,n)$}$ such that $\pi_j^i=x$ for any $x\in\{0,1,\ldots,n-1\}$. Therefore, we have $\Pi\in S^d_n$.
	
	{\noindent\bf Claim.}
	If $\Pi\in M_{d,n}$, then $\sum_{i=2}^d\pi_n^i=m_{d,n} \mbox{\ \ and\ \ }\sum_{i=2}^d s_i \equiv 0\text{ $(\md\,n)$}.$

    {\noindent \bf Proof of Claim.}
	By Corollary~\ref{iff_min_per}, we have $\sum_{i=2}^d \pi_n^{i}=\levSum(\Pi_n)=m_{d,n}$ and
	\[m_{d,n}=\levSum(\Pi_1)=\sum_{i=2}^d \pi_1^{i}\equiv\sum_{i=2}^d (\pi_n^i+s_i)=m_{d,n}+\sum_{i=2}^d s_i\text{ $(\md\,n)$},\]
	which gives the desired result.	

	Now we are in a position to prove the rest of this theorem.
	Note that for $k=1$ ($n=3$) statements (a) and (b) are consistent.

	\begin{enumerate}		
		\item[{\rm (a)}] For sufficiency, verification is straightforward. To prove necessity, assume that $\Pi\in M_{4,n}$.
		
		By the above claim, we have $n\mid (s_2+s_3+s_4)=2+s_4$, so $s_4=n-2$, which is co-prime with $n$ since $n$ is odd.
		
		Because $s_2=1$ and $\pi_n^2=n-1$, we have $\pi_j^2=j-1$ for $1\leq j\leq n$. Now, we only need to choose $(\pi_j^3,\pi_j^4)$ such that $3k=m_{4,n}=\pi_j^3+\pi_j^4+j-1$ for all $j$. Hence, $0\leq \pi_j^3\leq 3k-(j-1)$. Note that
		\begin{align*}
		\pi_j^3&=\left\{\begin{array}{ll}
		\pi_n^3+j,& 1\leq j \leq n-1-\pi_n^3;\vspace{0.15cm}\\
		\pi_n^3+j-n,& n-\pi_n^3\leq j \leq n.\\
		\end{array}\right.
		\end{align*}
		Therefore, in the case of $j=n-1-\pi_{n}^3$ and $j=n$, we have
		\begin{align*}
		\left\{\begin{array}{ll}
		\pi_{n-1-\pi_n^3}^3=\pi_n^3+n-1-\pi_n^3\leq 3k-(n-1-\pi_n^3-1);\vspace{0.15cm}\\
		\pi_n^3\leq 3k-(n-1).\\
		\end{array}\right.
		\end{align*}
		Finally, we obtain that $\pi_n^3$ equals to $k-1$ or $k$, and the item (a) is proved.
		
		\item[{\rm (b)}] For sufficiency, we only need to verify the case of $(\pi_n^3,\pi_n^4)=(0,k)$, as the other case can be obtained by swapping the last two rows of $\Pi$. In fact, if $s_3=s_4=k$ and $(\pi_n^3,\pi_n^4)=(0,k)$, then we have
		\begin{align*}
		\pi_j^3&=\left\{\begin{array}{ll}
		2k-\frac{j-2}{2},& \text{if $j$ is even;}\vspace{0.15cm}\\
		k-\frac{j-1}{2}, & \text{if $j$ is odd;}\\
		\end{array}\right. \text{\qquad and \qquad} \pi_j^4&\hspace{-1cm}=\left\{\begin{array}{ll}
		k-\frac{j}{2},& \text{if $j$ is even;}\vspace{0.15cm}\\
		2k-\frac{j-1}{2}, & \text{if $j$ is odd.}\\
		\end{array}\right.
		\end{align*}
		A straightforward verification shows that $\Pi\in M_{4,d}$.
		
		Conversely, assuming that $\Pi\in M_{4,d}$, we get $s_4=k$ and $0\leq \pi_j^3\leq 3k-(j-1)$ in the same way as in (a). In particular, $0\leq \pi_n^3\leq k$ and $0\leq \pi_{n-1}^3\leq k+1$. If $0\leq \pi_n^3\leq k-1$, then $k+1\geq\pi_{n-1}^3=\pi_n^3-k+n$. Hence, $\pi_n^3=0$ or $k$, and this completes the proof.
	\end{enumerate}	
\vspace{-0.83cm}
\end{proof}

For an illustration of the construction of the shift-form permutations in Theorem \ref{shiftform}, see the following example.
\begin{exam}
	Let $d=4$, $n=5$, $(\pi_n^2,\pi_n^3,\pi_n^4)=(4,0,2)$ and $(s_2,s_3,s_4)=(1,2,2)$. Then the array
	\begin{equation}\label{ex-min-max}
		\Pi=\left\{ \pi_j^i\right\}_{\begin{subarray}{l} 2 \leq i\leq d \\ 1\le j\le n\end{subarray}}=\left(\begin{array}{ccccc} 0 & 1 & 2 & 3 & 4 \\ 2 & 4 & 1 & 3 & 0\\ 4 & 1 & 3 & 0 & 2 \end{array}\right),\end{equation}
	where $ \pi_j^i\equiv\pi_n^i+j\cdot s_i \text{ $(\md\,n)$}$, is indeed a minimal $(4,5)$-permutation.
\end{exam}

\begin{rem}
	In the case of $d=4$, computer experiments show that all elements in $M_{4,3}$ and $M_{4,5}$ can be constructed using the method given in Theorem~\ref{shiftform}. However, there exist other permutations in $M_{4,7}$, for example,
	\[\Pi=\left(\begin{array}{ccccccc}
	0&1&2&3&4&5&6\\
	6& 4& 1& 5& 0& 2& 3\\
	3& 4& 6& 1& 5& 2& 0\\
	\end{array}\right)\in M_{4,7}.
	\]
\end{rem}

\subsection{Maximal permutations}\label{max-perms-sec}
In this subsection, similar to the previous one, we define maximal permutations and further determine their minimum level. Besides, we also establish several connections between these maximal and minimal permutations.

For any definition of $\level$, let $m^*_{d,n}=\max\{\level_{\min}(\Pi):\Pi\in S_n^d\}$. A \textit{maximal $(d,n)$-permutation} is a $(d,n)$-permutation $\Pi$ whose minimal entry in $\level(\Pi)$ is $m^*_{d,n}$. We denote by $M^*_{d,n}$ the set of all maximal $(d,n)$-permutations, i.e., $M^*_{d,n}=\{\Pi\in S_n^d:\level_{\min}(\Pi)=m^*_{d,n}\}$.

For $\levSum$, the case of $d=2$ is trivial and we have $m^*_{2,n}=0$ and $M^*_{2,n}=S_n$. For $d=3$, it is obvious that $m_{d,n}=m^*_{d,n}$ and $M_{d,n}=M^*_{d,n}$. For $d\geq 3$, we have the following result.

\begin{thm}\label{maximal-thm} We have that $m^*_{d,n}=\left\lfloor \frac{(d-1)(n-1)}{2} \right\rfloor$  for all $n\geq 1$ and $d\geq 3$.
\end{thm}

\begin{proof}
	We start by presenting an upper bound for $m^*_{d,n}$. Similar to the proof of Theorem~\ref{minimal-thm}, we have $\frac{n(n-1)(d-1)}{2}\geq n\cdot m^*_{d,n}$ for any $\Pi\in M^*_{d,n}$. Consequently,
	\begin{align}\label{upper-bound-on-m*dn}
	m^*_{d,n}\leq \frac{(d-1)(n-1)}{2}.
	\end{align}
	The remaining proof follows precisely the same steps as those in the proof of Theorem~\ref{minimal-thm}, where the minimal entries of the level vectors of $(d,n)$-permutations constructed are exactly $\left\lfloor \frac{(d-1)(n-1)}{2} \right\rfloor$.
\end{proof}

Now, we are ready to study the set $M^*_{d,n}$. The case of $n=1$ is trivial. For $n=2$, the set $M^*_{d,2}$ coincides with the set $M_{d,2}$, as shown in the next corollary.

\begin{cor}\label{2_max_per}
	If $n=2$ and $d\geq2$, then $M^*_{d,2}=M_{d,2}$.
\end{cor}
\begin{proof}
	If $\Pi\in M^*_{d,2}$, then we can assume that $\levSum(\Pi_1)=m^*_{d,2}=\left\lfloor \frac{d-1}{2} \right\rfloor$. Since the sum of $\levSum(\Pi_1)$ and $\levSum(\Pi_2)$ is the total sum of all entries in $\Pi$, i.e., $\levSum(\Pi_1)+\levSum(\Pi_2)=d-1$, we have that $\levSum(\Pi_2)=\left\lceil \frac{d-1}{2} \right\rceil=m_{d,2}$, implying that $\Pi\in M_{d,2}$. Hence, $M^*_{d,2}\subseteq M_{d,2}$. Justifying that $M_{d,2}\subseteq M^*_{d,2}$ can be done similarly.
\end{proof}

If either $n$ or $d$ is odd, not only do we have  $M^*_{d,n}=M_{d,n}$, but also the following stronger conclusions can be derived.

\begin{cor}\label{odd_max_per}
	Let $n\geq1$, $d\geq 2$. If either $n$ or $d$ is odd, then
	\begin{enumerate}
		\item[{\em (a)}] $m^*_{d,n}=m_{d,n}$.
		\item[{\em (b)}] $\Pi\in M^*_{d,n}$ if and only if $\levSum(\Pi_i)=m^*_{d,n}$ for $1\leq i\leq n$.
		\item[{\em (c)}] $M^*_{d,n}=M_{d,n}$.
	\end{enumerate}
\end{cor}
\begin{proof}
	\qquad
	\begin{enumerate}
		\item[(a)] Since either $n$ or $d$ is odd, we have $m^*_{d,n}=\left\lfloor \frac{(d-1)(n-1)}{2} \right\rfloor=\frac{(d-1)(n-1)}{2}=\left\lceil \frac{(d-1)(n-1)}{2} \right\rceil=m_{d,n}$.
		\item[(b)] We can follow exactly the same steps as those in the proof of Corollary~\ref{iff_min_per}, replacing $M_{d,n}$ with $M^*_{d,n}$,  $m_{d,n}$ with $m^*_{d,n}$, $\geq$ with $\leq$, $<$ with $>$, and the ceiling function with the floor function.
		\item[(c)] It follows from (a), (b) and Corollary~\ref{iff_min_per}.
	\end{enumerate}
\vspace{-0.77cm}
\end{proof}

\begin{rem}
	In general, $M^*_{d,n}$ may differ from $M_{d,n}$. For example, consider the permutation
	\begin{equation}\label{eq-Pi-min-not-max}
		\Pi=\left(\begin{array}{cccc}
			1 & 2 & 0 & 3\\
			1 & 0 & 3 & 2\\
			3 & 1 & 2 & 0
			\end{array}\right)\in S_4^4,
	\end{equation}
	where we find $\levSum_{\max}(\Pi)=5=m_{4,4}$ and $\levSum_{\min}(\Pi)=3<4=m^*_{4,4}$. Therefore, $\Pi$ is a minimal permutation but not a maximal one.
\end{rem}

While $M^*_{d,n}$ and $M_{d,n}$ can differ, we can always construct a bijection between these two sets. Therefore, $M^*_{d,n}$ and $M_{d,n}$ are equinumerous.

\begin{thm}\label{thm-min-max-equal}
	We have that $|M^*_{d,n}|=|M_{d,n}|$ for all $n\geq1$ and $d\geq2$.
\end{thm}
\begin{proof}
	It is sufficient to present a bijection between $M^*_{d,n}$ and $M_{d,n}$. Consider the map $\phi:S_n^d\rightarrow S_n^d$ such that $\phi(\Pi)$ is obtained from $\Pi\in S_n^d$ by taking the complement of each row in $\Pi$ (replacing an entry $i$ in a row by $n-1-i$ for all $i$). Clearly, $\phi$ is a bijection. We only need to show that $\phi(\Pi)\in M_{d,n}$ for any $\Pi\in M^*_{d,n}$ and $\phi^{-1}(\Pi)\in M^*_{d,n}$ for any $\Pi\in M_{d,n}$. Given a maximal $(d,n)$-permutation $\Pi$, we have $\levSum(\Pi_i)\geq m^*_{d,n}$ and there is an element, say $\Pi_{i_0}$, such that $\levSum(\Pi_{i_0})=m^*_{d,n}$. Then we have $\levSum(\phi(\Pi)_i)=(n-1)(d-1)-\levSum(\Pi_i)\leq (n-1)(d-1)-m^*_{d,n}=m_{d,n}$ (the last equality follows from Theorems~\ref{minimal-thm} and~\ref{maximal-thm}) and $\levSum(\phi(\Pi)_{i_0})=m_{d,n}$, so $\phi(\Pi)\in M_{d,n}$. It is analogous to prove that $\phi^{-1}(\Pi)\in M^*_{d,n}$ if $\Pi\in M_{d,n}$, so we omit the proof.
\end{proof}
{
We next give an example to illustrate the map $\phi$ in Theorem~\ref{thm-min-max-equal}.

\begin{exam}
	Let $\Pi$ be the minimal permutation in \eqref{eq-Pi-min-not-max}. Then the permutation
	\begin{equation*}
		\phi(\Pi)=\left(\begin{array}{cccc}
			2 & 1 & 3 & 0\\
			2 & 3 & 0 & 1\\
			0 & 2 & 1 & 3
		\end{array}\right)\in S_4^4,
	\end{equation*}
	where $\levSum_{\max}(\Pi)=6>m_{4,4}$ and $\levSum_{\min}(\Pi)=4=m^*_{4,4}$, is a maximal permutation but not a minimal one.
\end{exam}}
\subsection{Minimax permutations}\label{minimax-sec}
In this subsection, we explore the equivalent conditions for minimax permutations, which are defined as permutations that are both minimal and maximal.

We call a $(d,n)$-permutation a \textit{minimax $(d,n)$-permutation} if it is both minimal and maximal. Let $M^{\circ}_{d,n}$ denote the set $M_{d,n}\cap M^*_{d,n}$, which consists of all minimax $(d,n)$-permutations. Clearly, for $\levMax$, we have $M^{\circ}_{d,n}=M^*_{d,n}$. However, under $\levSum$, the situation is more complex. Corollaries~\ref{2_max_per} and~\ref{odd_max_per} imply that $M^{\circ}_{d,n}=M_{d,n}=M^*_{d,n}$ if $n=2$ or if at least one of $n$ or $d$ is odd. A complete classification of minimax permutations is given by the following theorem.

\begin{thm}
	Let $n\geq1$ and $d\geq2$. A $(d,n)$-permutation $\Pi$ is a minimax permutation if and only if one of the following two conditions holds:
	\begin{itemize}
	\item[{\rm (a)}]  $m_{d,n}=m^*_{d,n}$ and all elements of $\Pi$ are on level $m_{d,n}$;
	\item[{\rm (b)}]  $m_{d,n}=m^*_{d,n}+1$ and there are $\frac{n}{2}$ elements on level $m^*_{d,n}$ and $\frac{n}{2}$ elements on level $m_{d,n}$ in $\Pi$.
	\end{itemize}
\end{thm}

\begin{proof} If either $n$ or $d$ is odd, then $m_{d,n}=m^*_{d,n}$ and the statement follows from Corollary~\ref{odd_max_per}.

	Assume that $n$ and $d$ are even with $n=2s$ for $s\geq1$. Then we have $m_{d,n}=m^*_{d,n}+1$. The sufficiency is evident by the definition of a minimax permutation. For the necessity, assume that $\Pi\in M^{\circ}_{d,n}$. Since $m^*_{d,n}=\levSum_{\min}(\Pi)\leq\levSum_{\max}(\Pi)= m_{d,n}$, the level of each element of $\Pi$ is either $m^*_{d,n}$ or $m_{d,n}$. Now we only need to show that the numbers of elements on levels $m^*_{d,n}$ and $m_{d,n}$ are as claimed. Suppose that there are $s-t$ elements on level $m^*_{d,n}$ and $s+t$ elements on level $m_{d,n}$ for an integer $t\neq0$. Considering the total sum of all entries in $\Pi$, we have $s(2s-1)(d-1)=(s-t)m^*_{d,n}+(s+t)m_{d,n}=s(m^*_{d,n}+m_{d,n})+t(m_{d,n}-m^*_{d,n})=s(2s-1)(d-1)+t\neq s(2s-1)(d-1)$, which is a contradiction.
\end{proof}

Note that the permutation given in \eqref{ex-min-max} is, in fact, a minimax permutation. Furthermore, all minimal
$(4,n)$-permutations constructed in Theorem~\ref{shiftform} are minimax permutations, since $m_{d,n}=m^*_{d,n}$ for odd $n\ge3$.

\subsection{$c$-bounded permutations}\label{c-bounded-perms-sec}
This subsection is devoted to the introduction of the $c$-bounded permutations. We present the connections between these permutations and other combinatorial objects in the OEIS~\cite{oeis}, along with a formula that unifies them within the same context.

For any $\Pi\in S_n^d$ and definition of $\level$, $\level_{\max}(\Pi) \geq m_{d,n}$. Introducing a cap $c\geq 0$ to the level of permutations in $S_n^d$, we obtain the set of {\em $c$-bounded $(d,n)$-permutations}:
$$M^{+c}_{d,n}=\{\Pi\in S_n^d: m_{d,n}\leq \level_{\max}(\Pi) \leq m_{d,n}+ c\}.$$
Note that the condition $\level_{\max}(\Pi)\geq m_{d,n}$ is trivial, so we have the equivalent definition $$M^{+c}_{d,n}=\{\Pi\in S_n^d:\level(\Pi_i) \leq m_{d,n}+ c,\text{ for all } 1\leq i\leq n\}.$$ Clearly, for $c=0$,  $M^{+0}_{d,n}=M_{d,n}$.
{For general cases, see the following theorem for the exact cardinality of $M^{+c}_{3,n}$.}
\begin{thm}\label{cardinality-c-bounded-thm}
We have $|M^{+c}_{3,n}|=c!(c+1)^{n-c}n!$ for all $0\leq c\leq n-1$.
\end{thm}

\begin{proof}
We only need to show that the number of ways to choose elements of $\Pi\in S_n^3$ is $c!(c+1)^{n-c}$; the rest will be given by permuting the elements in $n!$ ways. W.l.o.g., $\pi^2_1\pi^2_2\ldots\pi^2_n=01\ldots (n-1)$. Consider the element $(n-1,\pi^3_{n})^T$. Since $\levSum(\Pi_{n})\leq n-1+c$ and hence $\pi^3_{n}\in\{0,1,\ldots, c\}$, there are $c+1$ choices for $\pi^3_{n}$. Fix $\pi^3_{n}$ and consider the  element $(n-2,\pi^3_{n-1})^T$. Since $\pi^3_{n-1}\neq \pi^3_{n}$ and $\pi^3_{n-1}\in\{0,1,\ldots, c+1\}$, again, we have $c+1$ choices for $\pi^3_{n-1}$. Fix $\pi^3_{n-1}$ and consider the element $(n-3,\pi^3_{n-2})^T$. Since $\pi^3_{n-2}\neq \pi^3_{n}$ and $\pi^3_{n-2}\neq \pi^3_{n-1}$ and $\pi^3_{n-2}\in\{0,1,\ldots, c+2\}$, again, we have $c+1$ choices for $\pi^3_{n-2}$. And so on until we fix the element $(c,\pi^3_{c+1})^T$. The number of choices for $(c-1,\pi^3_{c})^T$ will then be $c$, for $(c-2,\pi^3_{c-1})^T$ it will be $c-1$, and so on until we will make the unique choice for $(0,\pi^3_{1})^T$, completing our proof.
\end{proof}

Interestingly, in the case of $d=3$, $M^{+1}_{3,n}$ gives a new combinatorial interpretation for the sequence A002866 and $M^{+2}_{3,n}$ introduces a previously unknown combinatorial interpretation for the sequence A052700 in \cite{oeis}.

Removing $n!$ (i.e., scaling by $n!$) and considering a canonical representation of permutations in question, we obtain $c!(c+1)^{n-c}$ objects. The formulae for $c=0$ and $c=1$ are trivial and simple, respectively. For $c=2$, $c=3$ and $c=4$, but not for $c=5$, we also get matches in the OEIS \cite{oeis}.
\begin{itemize}
\item For $c=2$,  $|\Ca(M^{+2}_{3,n})|=2\cdot 3^{n-2}$, given by sequence A008776 in \cite{oeis}. There are several combinatorial interpretations of these numbers, for example:
\begin{itemize}
\item Compositions of $n-1$ where there are two types of each natural number, say, barred and non-barred. For instance, the 6 compositions of $2$ are $\{2,\bar{2},1+1,\bar{1}+1,1+\bar{1},\bar{1}+\bar{1}\}$.
\item Functions $f:\{1,2,\ldots,n-1\}\mapsto\{1,2,3\}$ such that for fixed $x\in\{1,2,\ldots,n-1\}$ and fixed $y\in\{1,2,3\}$, we have $f(x) \neq y$.
\end{itemize}
\item For $c=3$,  $|\Ca(M^{+3}_{3,n})|=6\cdot 4^{n-3}$, given by sequence A002023 in \cite{oeis}. The known combinatorial interpretations are related to the {\em Sierpinski tetrahedron}, also known as {\em tetrix}. In particular, the sequence counts the number of edges in the $(n-2)$-Sierpinski tetrahedron graph.
\item For $c=4$ and $|\Ca(M^{+4}_{3,n})|=24\cdot 5^{n-4}$, given by sequence A235702 in \cite{oeis}. This sequence is related to the {\em Pisano periods}, which are periods of {\em Fibonacci numbers} modulo $n$. However, our objects give the first (direct) combinatorial interpretation for these numbers.
\end{itemize}
Once again, we offer new combinatorial interpretations that unify sequences A008776, A002023 and A235702 from \cite{oeis}.

\subsection{$k$-plateaux among permutations in $S_n^3$}\label{sec-occurrences-of-patterns}
This subsection is devoted to the enumeration of occurrences of $k$-plateaux in all $(3,n)$-permutations for all $2\le k\le n$.

We begin by giving some definitions and notations. For any $n\geq1$, $d\geq 2$, and $0\leq \ell\leq (n-1)(d-1)$, let $e_{d,n,\ell}$ denote the number of distinct elements on level $\ell$ among all permutations in $S^d_n$, and let $E_{d,n,\ell}$ be the set defined as $$\{(x_1,x_2,\ldots,x_{d-1})\in\mathbb{Z}^{d-1}:x_1+x_2+\cdots+x_{d-1}=\ell\text{ and } 0\leq x_i\leq n-1\text{ for }1\leq i \leq d-1\}.$$

In what follows, for a function $f(x)$, $[x^n]f(x)$ denotes the coefficient of $x^n$ in the Taylor expansion of $f(x)$ around 0.
The next lemma provides several basic properties of $e_{d,n,\ell}$. Since these formulae can be derived through technical manipulations, we omit their proof here.

\begin{lem}\label{distinct elements of level l} We have
	\begin{enumerate}
		\item[{\em (a)}] $e_{d,n,\ell}=|E_{d,n,\ell}|=[x^\ell]\left(\frac{x^n-1}{x-1}\right)^{d-1}$.
		\item[{\em (b)}] $e_{d,n,\ell}=e_{d,n,(d-1)(n-1)-\ell}$ for $0\leq \ell\leq (n-1)(d-1)$.
		\item[{\em (c)}] $e_{d,n,\ell}={\ell+d-2\choose d-2}$ for $0\leq \ell\leq n-1$.
	\end{enumerate}
\end{lem}

To study the occurrences of patterns in all $(d,n)$-permutations, we need to examine the repetition of each element, as described in the following proposition.

\begin{prop}\label{distr-elements} Each element occurs $n!((n-1)!)^{d-2}$ times among all $(d,n)$-permutations, and the number of elements on level $\ell$, where $0\leq \ell\leq (n-1)(d-1)$, among all permutations in $S^d_n$ is $n!((n-1)!)^{d-2}e_{d,n,\ell}$.
\end{prop}

\begin{proof} It is sufficient to prove the first assertion. First, note that the elements are uniformly distributed among all permutations in $S^d_n$. Indeed, for any elements $(x_1,\ldots,x_{d-1})^T$ and $(y_1,\ldots,y_{d-1})^T$, we can consider a bijection on $S^d_n$ that swaps, in each row $i$ of each permutation, $x_i$ and $y_i$.

Now, the number of distinct elements is $n^{d-1}$ and the total number of elements in $S^d_n$ is $n(n!)^{d-1}$. Since the elements are uniformly distributed, each distinct element occurs $n(n!)^{d-1}/n^{d-1}=n!((n-1)!)^{d-2}$ times.
\end{proof}

Using a similar approach as in Theorem \ref{distr-k-plateau-all-2}, we obtain the following formula for the total number of $k$-plateaux, where $A(n,m)$ is defined as previously.
\begin{thm}\label{distr-k-plateau-all} Let $n\geq2$ and $2\leq k\leq n$. The total number of $k$-plateaux (occurrences of $\underline{11\ldots1}$ with $k$ $1$'s) among all permutations in $S^3_n$ is given by
	\begin{equation}\label{total-occ-k-plateau}
	(n-k+1)!(n-k)!\left(\frac{n!}{(n-k)!}+2\sum_{\ell=k-1}^{n-2}\frac{(\ell+1)!}{(\ell+1-k)!}\right).
	\end{equation}
\end{thm}

\begin{proof}
	First, we show that there are
	\begin{align}\label{number of k-plateau}
	\sum_{\ell=0}^{2(n-1)}A(e_{3,n,\ell},k)
	\end{align}
	distinct $k$-tuples of elements counted as $k$-plateaux in permutations. To obtain a $k$-plateau tuple occurring in some permutations, we simply need to choose $k$ different elements at the same level in order, ensuring they have distinct entries in each row. Thus, there are $A(e_{3,n,\ell},k)$ $k$-plateau tuples on level $\ell$. (Note that this observation cannot be extended to $d>3$ because then we cannot guarantee that the elements in each row in $k$ tuples selected from $E_{d,n,\ell}$ are distinct.) By summing over all possible levels, the assertion is proven.
	
	Next, we compute the expression~\eqref{number of k-plateau} by using Lemma~\ref{distinct elements of level l}.
	\begin{align*}
	\sum_{\ell=0}^{2(n-1)}A(e_{3,n,\ell},k)&=\sum_{\ell=0}^{n-2}A(e_{3,n,\ell},k)+A(e_{3,n,n-1},k)+\sum_{\ell=n}^{2(n-1)}A(e_{3,n,\ell},k)\\
	&=\sum_{\ell=0}^{n-2}A(e_{3,n,\ell},k)+A(e_{3,n,n-1},k)+\sum_{\ell=n}^{2(n-1)}A(e_{3,n,2(n-1)-\ell},k)\\
	&=A(e_{3,n,n-1},k)+2\sum_{\ell=0}^{n-2}A(e_{3,n,\ell},k)\\
	&=\frac{n!}{(n-k)!}+2\sum_{\ell=k-1}^{n-2}\frac{(\ell+1)!}{(\ell+1-k)!}.
	\end{align*}
	
	Finally, we need to demonstrate that any $k$-plateau tuple occurs in $(n-k+1)((n-k)!)^2$ different permutations. This is because, when constructing a permutation with a fixed $k$-plateau tuple, we have $(n-k+1)$ choices for the position of the $k$-plateau tuple and $((n-k)!)^2$ choices for arranging the remaining entries.
\end{proof}

By applying Theorem \ref{distr-k-plateau-all}, we can derive the following result for $(3,n)$-permutations, which surprisingly coincides with Corollary \ref{cor-plateau-S_n^3-levMax}.

\begin{cor}\label{distr-asc-des-plateau} Let $n\geq2$. The total number of plateaux (occurrences of $\underline{11}$) among all permutations in $S^3_n$ is given by
	\begin{equation}\label{total-occ-plateau} \frac{(2n-1)n!(n-1)!}{3}.\end{equation} The total number of ascents (the same as descents) in $S^3_n$ is given by
	\begin{equation}\label{total-occ-sum-asc} \frac{(3n^2-5n+1)n!(n-1)!}{6}.\end{equation}
\end{cor}

\begin{proof}
	The expression~\eqref{total-occ-plateau} is obtained by substituting $k$ with $2$ in \eqref{total-occ-k-plateau}.
	
For the second assertion, the proof follows the same reasoning as in Corollary~\ref{cor-plateau-S_n^3-levMax}, and this completes the proof.
\end{proof}

\section{Conclusion}\label{open-questions-sec}

We conclude our paper with several open problems related to the most interesting results of our studies.

Under $\levMax$, we give a new combinatorial interpretation for the Springer numbers by introducing weakly increasing $3$-dimensional permutations. However, our two methods rely on certain known generating functions and are not purely combinatorial. Since the Springer numbers have a rich background and numerous combinatorial interpretations, it is natural to consider whether there exist bijections between the weakly increasing $3$-dimensional permutations and the combinatorial objects counted by the Springer numbers. In particular, the following question is of interest.

\begin{prob}
	Find a bijective proof of Theorem \ref{thm-W}.
\end{prob}

Under $\levSum$, we examine the maximal levels of minimal permutations. Through a special construction, we characterize certain $4$-dimensional permutations of odd length as minimal permutations, as described in Theorem \ref{shiftform}. Interestingly, all elements of each minimal permutation under this condition must share the same level. However, enumerating these minimal permutations presents a significant challenge.

\begin{prob}
	Let $d=4$ and $n$ be an odd integer. Find an explicit formula for the numbers $|M_{d,\,n}|$.
\end{prob}

Although many enumerative results for the same patterns under $\levMax$ and $\levSum$ differ, Corollaries~\ref{cor-plateau-S_n^3-levMax} and~\ref{distr-asc-des-plateau} yield the same total count of plateaux across all permutations in $S_n^3$. This observation leads to the following problem.

\begin{prob}
	Find a bijection between the plateaux under $\levMax$ and $\levSum$ in all $(3,n)$-permutations.
\end{prob}
Throughout this paper, due to the complexity of handling high-dimensional permutations, we focus primarily on enumerative results for the 3- and 4-dimensional cases. Extending these results to permutations of higher dimensions is a natural next step. Additionally, as another direction for further research, one could explore alternative definitions for the level of an element in a multi-dimensional permutation and extend the enumerative and structural results presented in this paper to these new definitions.

\bibliographystyle{plain}

\bibliography{patterns-multi-dim-perms}

\begin{thebibliography}{10}

\bibitem{Arnold}
V.I. Arnol'd.
\newblock The calculus of snakes and the combinatorics of {B}ernoulli, {E}uler
  and {S}pringer numbers of {C}oxeter groups.
\newblock {\em Russ. Math. Surv.}, 47(1):1--51, 1992.

\bibitem{AM2010}
A.~Asinowski and T.~Mansour.
\newblock Separable {$d$}-permutations and guillotine partitions.
\newblock {\em Ann. Comb.}, 14(1):17--43, 2010.

\bibitem{AKLPT}
S.~Avgustinovich, S.~Kitaev, J.~Liese, V.~Potapov, and A.~Taranenko.
\newblock Singleton mesh patterns in multidimensional permutations.
\newblock {\em J. Combin. Theory Ser. A}, 201:Paper No. 105801, 23, 2024.

\bibitem{Branden-2015}
P.~Br\"{a}nd\'{e}n.
\newblock Unimodality, log-concavity, real-rootedness and beyond.
\newblock In {\em Handbook of enumerative combinatorics}, Discrete Math. Appl.
  (Boca Raton), pages 437--483. CRC Press, Boca Raton, FL, 2015.

\bibitem{Brenti-1994}
F.~Brenti.
\newblock Log-concave and unimodal sequences in algebra, combinatorics, and
  geometry: an update.
\newblock In {\em Jerusalem combinatorics '93}, volume 178 of {\em Contemp.
  Math.}, pages 71--89. Amer. Math. Soc., Providence, RI, 1994.

\bibitem{BurEli}
A.~Burstein and S.~Elizalde.
\newblock Total occurrence statistics on restricted permutations.
\newblock {\em Pure Math. Appl.}, 24(2):103--123, 2013.

\bibitem{Chen1993}
W.Y.C. Chen.
\newblock Context-free grammars, differential operators and formal power
  series.
\newblock volume 117, pages 113--129. 1993.
\newblock Conference on Formal Power Series and Algebraic Combinatorics
  (Bordeaux, 1991).

\bibitem{Chen-et-al}
W.Y.C. Chen, N.J.Y. Fan, and J.Y.T. Jia.
\newblock Labeled ballot paths and the {S}pringer numbers.
\newblock {\em SIAM J. Discrete Math.}, 25(4):1530--1546, 2011.

\bibitem{DB1962}
F.N. David and D.E. Barton.
\newblock {\em Combinatorial {C}hance}.
\newblock Hafner Publishing Co., New York, 1962.

\bibitem{EriLin}
K.~Eriksson and S.~Linusson.
\newblock A combinatorial theory of higher-dimensional permutation arrays.
\newblock {\em Adv. in Appl. Math.}, 25(2):194--211, 2000.

\bibitem{FriMan}
S.~Fried and T.~Mansour.
\newblock The total number of descents and levels in (cyclic) tensor words.
\newblock {\em Discrete Math. Lett.}, 13:44--49, 2024.

\bibitem{Gessel1977}
I.M. Gessel.
\newblock {\em G{enerating} {Functions} {and} {Enumeration} {of} {Sequences}}.
\newblock ProQuest LLC, Ann Arbor, MI, 1977.
\newblock Thesis (Ph.D.)--Massachusetts Institute of Technology.

\bibitem{Gla1898}
J.W.L. Glaisher.
\newblock On the {B}ernoullian function.
\newblock {\em Quart. J. Pure Appl. Math.}, 29:1--168, 1898.

\bibitem{Gla1914}
J.W.L. Glaisher.
\newblock On the coefficients in the expansions of $\cos x/\cos 2x$ and $\sin
  x/\cos 2x$.
\newblock {\em Quart. J. Pure Appl. Math.}, 45:187--222, 1914.

\bibitem{HKZ2024}
T.~Han, S.~Kitaev, and P.B. Zhang.
\newblock {D}istribution of maxima and minima statistics on alternating
  permutations, {S}pringer numbers, and avoidance of flat {POP}s.
\newblock arXiv:2408.12865, 2024.

\bibitem{HeuMan}
S.~Heubach and T.~Mansour.
\newblock {\em Combinatorics of {C}ompositions and {W}ords}.
\newblock Discrete Mathematics and its Applications (Boca Raton). CRC Press,
  Boca Raton, FL, 2010.

\bibitem{J-V2014}
M.~Josuat-Verg\`es.
\newblock Enumeration of snakes and cycle-alternating permutations.
\newblock {\em Australas. J. Combin.}, 60:279--305, 2014.

\bibitem{JuSeo2012}
H.K. Ju and S.~Seo.
\newblock Enumeration of {$(0,1)$}-matrices avoiding some {$2\times2$}
  matrices.
\newblock {\em Discrete Math.}, 312(16):2473--2481, 2012.

\bibitem{Sergey-2006}
S.~Kitaev.
\newblock Introduction to partially ordered patterns.
\newblock {\em Discrete Appl. Math.}, 155(8):929--944, 2007.

\bibitem{Kitaev2011}
S.~Kitaev.
\newblock {\em Patterns in {P}ermutations and {W}ords}.
\newblock Monographs in Theoretical Computer Science. An EATCS Series.
  Springer, Heidelberg, 2011.
\newblock With a foreword by Jeffrey B. Remmel.

\bibitem{KitManVel}
S.~Kitaev, T.~Mansour, and A.~Vella.
\newblock Pattern avoidance in matrices.
\newblock {\em J. Integer Seq.}, 8(2):Article 05.2.2, 16, 2005.

\bibitem{Linial2014}
N.~Linial and Z.~Luria.
\newblock An upper bound on the number of high-dimensional permutations.
\newblock {\em Combinatorica}, 34(4):471--486, 2014.

\bibitem{Linial2019}
N.~Linial, T.~Pitassi, and A.~Shraibman.
\newblock On the communication complexity of high-dimensional permutations.
\newblock In {\em 10th {I}nnovations in {T}heoretical {C}omputer {S}cience},
  volume 124 of {\em LIPIcs. Leibniz Int. Proc. Inform.}, pages Art. No. 54,
  20. Schloss Dagstuhl. Leibniz-Zent. Inform., Wadern, 2019.

\bibitem{Linial2018}
N.~Linial and M.~Simkin.
\newblock Monotone subsequences in high-dimensional permutations.
\newblock {\em Combin. Probab. Comput.}, 27(1):69--83, 2018.

\bibitem{Liu-Wang-2006}
L.L. Liu and Y.~Wang.
\newblock A unified approach to polynomial sequences with only real zeros.
\newblock {\em Adv. in Appl. Math.}, 38(4):542--560, 2007.

\bibitem{MA2012}
S.M. Ma.
\newblock Derivative polynomials and enumeration of permutations by number of
  interior and left peaks.
\newblock {\em Discrete Math.}, 312(2):405--412, 2012.

\bibitem{oeis}
N.J.A. Sloane et~al.
\newblock The {O}n-line {E}ncyclopedia of {I}nteger {S}equences.
\newblock {\em Published electronically at \url{https://oeis. org}}, 2018.

\bibitem{Sokal2020}
A.D. Sokal.
\newblock The {E}uler and {S}pringer numbers as moment sequences.
\newblock {\em Expo. Math.}, 38(1):1--26, 2020.

\bibitem{Springer1971}
T.A. Springer.
\newblock Remarks on a combinatorial problem.
\newblock {\em Nieuw Arch. Wisk. (3)}, 19:30--36, 1971.

\bibitem{Stanley-log-1989}
R.P. Stanley.
\newblock Log-concave and unimodal sequences in algebra, combinatorics, and
  geometry.
\newblock In {\em Graph theory and its applications: {E}ast and {W}est
  ({J}inan, 1986)}, volume 576 of {\em Ann. New York Acad. Sci.}, pages
  500--535. New York Acad. Sci., New York, 1989.

\bibitem{ZG2007}
H.~Zhang and D.~Gildea.
\newblock Enumeration of factorizable multi-dimensional permutations.
\newblock {\em J. Integer Seq.}, 10(5):Article 07.5.8, 18, 2007.

\bibitem{Zhuang2016}
Y.~Zhuang.
\newblock Counting permutations by runs.
\newblock {\em J. Combin. Theory Ser. A}, 142:147--176, 2016.

\end{thebibliography}

\end{document}